\documentclass[8pt,epsfig]{amsart}
\usepackage{amssymb, adjustbox, enumerate, amsbsy, stmaryrd}
\usepackage[dvipsnames]{xcolor}
\usepackage{amsmath,wasysym}
\usepackage{comment}
\usepackage{caption}
\usepackage{stackrel}
\usepackage[mathscr]{eucal}
\usepackage{amsthm}
\usepackage{graphicx}
\usepackage {amscd}
\usepackage {epic}
\usepackage[all,color]{xy}
\usepackage[alphabetic]{amsrefs}
\usepackage{extarrows}  % \xleftrightarrow
\usepackage{color}

\usepackage{graphpap, color}

\usepackage[mathscr]{eucal}
\usepackage{mathrsfs}%,mathabx}
\usepackage{graphicx}

\usepackage{tikz}

\newtheorem{prop}{Proposition}[section]

\newtheorem{coro}[prop]{Corollary}
\newtheorem{defi}[prop]{Definition}

\newtheorem{exam}[prop]{Example}

\newtheorem{lemm}[prop]{Lemma}
\newtheorem{pf-thm}[prop]{proof of theorem}
\newtheorem{rema}[prop]{Remark}

\newtheorem{theo}[prop]{Theorem}

\newtheorem*{ack}{Acknowledgments}

\def\cC{\mathcal{C}}

\def\cH{\mathcal{H}}
\def\cJ{\mathcal{J}}
\def\cK{\mathcal{K}}
\def\cM{\mathcal{M}}
\def\cO{\mathcal{O}}

\def\cW{\mathcal{W}}
\def\cY{\mathcal{Y}}
\def\cX{\mathcal{X}}
\def\cZ{\mathcal{Z}}
\def\cYM{\mathcal{YM}}

\def\CC{\mathbb{C}}

\def\RR{\mathbb{R}}

\def\dd{\mathrm{partial}}

\def\diff{\mathrm{Diffeo}}

\def\Emap{\mathrm{EMap}}
\def\End{\mathrm{End}}

\def\ham{\mathrm{Ham}}

\def\Kmap{\mathrm{KMap}}

\def\map{\mathrm{Map}}

\def\proj{\mathrm{Proj}}
\def\psh{\mathrm{PSH}}

\def\Ric{\mathrm{Ric}}

\def\sym{\mathrm{Sym}}

\def\tr{\mathrm{Tr}}

\def\vol{\mathrm{vol}}

\def\Ad{\mathrm{Ad}}

\def\dbar{\bar \partial}
\def\dd{\partial}

\newcommand{\ddbar}{\sqrt{-1}\partial\bar\partial}

\input{xy}
\xyoption{all}

\begin{document}

\title{Moment map for coupled equations of K{\"a}hler forms and curvature}

\author{King Leung Lee}
\address{Instituto de Ciencias Matem\'{a}ticas, ICMAT, C. Nicol\'{a}s Cabrera, 13-15, 28049 Madrid, Spain}
\email{king.lee@icmat.es}

\date{\today}
\maketitle
\begin{abstract}
In this paper we introduce two new systems of equations in K{\"a}hler geometry: The coupled p equation and the generalized coupled cscK equation. We motivate the equations from the moment map pictures, prove the uniqueness of solutions and find out the obstructions to the  solutions for the second equation. We also point out the connections between the coupled cscK equation, the coupled K{\"a}hler Yang-Mills equations and the deformed Hermitian Yang-Mills equation.

Moreover, using this moment map, we can show the Mabuchi functional for the generalized coupled cscK equation, and a special case of the coupled K{\"a}hler Yang-Mills equations and the deformed Hermitian Yang-Mills equation are convex along the smooth geodesic, which is different from the one using the moment map picture from the gauge group. In our case, the geodesic is given by the natural metric on the product of smooth K\"{a}hler potential $\cK(X,\omega_0)\times \cdots \times \cK(X,\omega_k)$.
\end{abstract}

\section{Introduction}
\subsection{Motivation}
Over the years, many important equations in complex geometry have been given moment map interpretations. A few examples of equations with moment map interpretations are the cscK equation (\cite{Don00} and \cite{Fuj92}), the coupled K{\"a}hler Yang-Mills equation (\cite{AGG13}) and the coupled constant scalar curvature (\cite{DaPi19}). In this paper, we combine the moment maps for the latter two together with some new ideas to define a new type of canonical metrics. We begin by recalling the definition of the coupled Kähler-Yang-Mills equation.

\begin{defi}[\cite{AGG13}]
Let $P$ be a principal $U(k)$ bundle on a K{\"a}hler manifold $(X,\omega_X)$,  $A$ be a connection on $P$, and $F_A$ be the curvature which is an $Lie(G)$-valued 2 form. Then the coupled K{\"a}hler equation is given by 
\[\begin{matrix}
\alpha_0S_{g}+\alpha_1 \bigwedge^2 F_A\wedge F_A &=& c\\
\bigwedge F_A &=& z
\end{matrix}\]
where $z\in Lie(G)$ is invariant under the adjoint $U(k)$ action and $c$ is a constant, which depended on the topological constraint on $P$ and $\alpha_0,\alpha_1,[\omega]$. Also, we need the integrability condition
\[F_A^{0,2} = 0.\]
\end{defi}
If  $P=U(1)^k$, then the Lie algebra is $u(1)\oplus\cdots \oplus u(1)$, and the $F_A$ can be represented as \[F_A=\omega_1+\cdots+\omega_k,\] where $\omega_1,...,\omega_k$ are $L$-valued K{\"a}hler forms on $X$, which can be realized as K{\"a}hler forms on $X$. In this special case,  the moment map equation is given by 
\[\begin{matrix}
\displaystyle\alpha_0S_{g}+\alpha_1 \sum_{i=1}^k\dfrac{\omega_i^{2}\wedge \omega_0^{n-2}}{\omega_0^n}&=& c_0\\
\tr_{\omega_0}(\omega_1)&=&c_1\\
\vdots &=&\vdots\\
\tr_{\omega_0}(\omega_k)&=&c_k.
\end{matrix}.\]
 In \cite{HuNy16},   Hultgren and Witt Nystr\"{o}m introduced another type of canonical metrics: the coupled K{\"a}hler Einstein equation. This equation was later generalized by Datar and Pingali (\cite{DaPi19}) to the coupled cscK equation:
 \begin{defi}[\cite{DaPi19}]
 Let $X$ be a K{\"a}hler manifold and $\omega_0,...,\omega_k$ be K{\"a}hler forms on $X$, and let $\displaystyle \omega=\sum_{i=0}^k\omega_i$. Then the coupled cscK equation is given by 
 \[\begin{matrix}
 \dfrac{\omega_0^n}{\vol(\omega_0)}=\cdots= \dfrac{\omega_k^n}{\vol(\omega_k)}\\
 S_{\omega_0}=\tr_{\omega_0}\omega+c,
 \end{matrix}\]
 here $c$ is the topological constant depending on the class of $\omega_i$ and $\Ric(\omega)$. If $c=0$, then this reduces to the coupled K{\"a}hler-Einstein equation. 
 \end{defi}
 
 Both the cKYM equation and the ccscK equation are moment map equations. For both setups, the domains are in a subspace of $\cY\subset\cJ\times \mathcal{A}$, which for each $(J,A)$, $A\in \Omega_J^{1,1}(ad(P))$ (See \cite{AGG13}, \cite{DaPi19}). Notice that in order to get the topological constraint, the setups are in the complexifed orbit. But we can study deformation of ths solutions if we consider the bigger subspace $\cY$. When $P$ is a  principal $U(1)$ bundle, the moment map for the coupled K{\"a}hler Yang-Mills equation is 
 \begin{align*}
\mu_{cKYM}(J,A)(\varphi,\xi)=&\int_X\varphi\left(c-S(J)-\alpha_2\dfrac{\omega_X^{n-2}\wedge\omega_A^2}{\omega_X^n}+\alpha_1\dfrac{\omega_X^{n-1}\wedge\omega_A}{\omega_X^n}\right)\dfrac{\omega_X^n}{n!}\\
&+\int_X\theta_A\xi \wedge\left(\alpha_1z-\alpha_2\dfrac{\omega_X^{n-1}\wedge\omega_A}{\omega_X^n}\right)\omega_X^n;\end{align*}
 and the moment map for coupled cscK equation is 
 \[\mu_{ccscK}(J,A)(H_{\xi,0}, H_{\xi,A})=\int_XH_{\xi,0}\left(c-S(J)+\dfrac{\omega_X^{n-1}\wedge\omega_A}{\omega_X^n}\right)+\int_X H_{\xi,A} \left(\dfrac{\omega_A^n}{\omega_X^n}-z\right)\]
 \\
 
We will now explain how both these moment maps can be constructed using the moment map for the cscK metrics together with a new construction involving maps between two symplectic manifolds  $(X,\omega_X)$ and $(Y,\omega_Y)$ which are diffeomorphic to each others. 
\begin{defi}[Defintion \ref{eqn: main moment map}]
 We denote the map
 \[\mu_p:\map(X,Y;p)^+ \rightarrow Lie(\ham(X,\omega_X)\times \ham(Y,\omega_Y))^*\] by
\begin{equation}
\mu_{p,\omega_X\omega_Y}(f)=\dfrac{n}{n-p}\left(c_1\dfrac{\omega_X^n}{n!}-\dfrac{\omega_X^{n-p-1}\wedge f^*\omega_Y^{p+1}}{(n-p-1)!(p+1)!},c_2\dfrac{\omega_Y^n}{n!}- \dfrac{f_*\omega_X^{n-p}\wedge\omega_Y^p}{(n-p)!p!}\right),
\end{equation}
where $\displaystyle c_1=\dfrac{\int_X\omega_X^{n-p-1}\wedge f^*\omega_Y^{p+1}}{\int_X\omega_X^n}$, $\displaystyle c_2=\dfrac{\int_Yf_*\omega_X^{n-p}\wedge\omega_Y^p}{\int_Y\omega_Y^n}$. 
\end{defi}

We also have the classical moment map: Denote $\cJ_{int}(X)$ be the space of all integrable almost complex structure, and  let \[\cJ(X,\omega_0):=\{J \in \cJ_{int}(X)| \omega_0(\bullet, \bullet)= \omega_0(J\bullet, J\bullet), \omega_0(\bullet, J\bullet)>0\}\] be the space of integrable almost complex structure compactible with $\omega_0$. The metric $g_J:=\omega(\bullet, J\bullet)$ induces a pairing on $T_J\cJ(X,\omega_0)$, and \[\Omega_{\cJ}(\delta_1J,\delta_2J):=\int_X g_J(\delta_1J,\delta_2J)\dfrac{\omega^n}{n!}.\] Then the map
\[\mu_{\cJ}(J)=(S_{J}-\underline{S_J})\dfrac{\omega_0^n}{n!}=\Ric(X,J)\wedge\dfrac{\omega_0^{n-1}}{(n-1)!}-S\dfrac{\omega_n}{n!}\] is a moment map corresponding to $(\cJ(X,\omega_0), \Omega_{\cJ})$ (see \cite{Don97}, \cite{Fuj92}), where 
\[\underline{S_J}=\dfrac{1}{Vol(X,\omega^n)}\int_X S_J\dfrac{\omega_0^n}{n!}\] is the average of $S_J$.\\  

As $X$ is diffeomorphic to $Y$, if we take $\omega_A=f^*\omega_Y$, then under suitable domain, 
\begin{enumerate}
	\item $\mu_{cKYM}(J,A)=0$ iff $\mu_{\cJ}(J)+c_1\mu_1^*(f)-c_2\mu_0^*(f)=0$ for some suitable constant $c_1,c_2$;
	\item $\mu_{ccscK}(J,A)=0$ iff $\mu_{\cJ}(J)+c\mu_0(f)=0$ for some suitable constant $c$.
\end{enumerate}

Notice that with a suitable choice of domain and symplectic form, the sum of two moment maps can also be a moment map.  Therefore, we unify the ccscK equations and coupled K{\"a}hler Yang Mills equation into one general moment map setup, namely,  the sum of different moment maps $\mu_p$ with the standard moment map $\mu_{\cJ}$. Moreover, using the same idea, we reconstruct the moment map for deformed Hermitian Yang Mills equation (dHYM) (see \cite{CXY17}) and the coupled dHYM (\cite{ScSt19}) in section 3.4. \\

\subsection{Construction}
We will now explain how to choose the symplectic form, the domain and the range to make the sum of two moment maps a moment map in general by considering the construction for the moment map $\mu_{\cJ}+\mu_0$, i.e, the moment map for ccscK, as an example.
\begin{enumerate}
	\item[Step 1] Define \[\mu_{\cJ,0}: \cJ(X,\omega_0)\times\map(X,X)\rightarrow  Lie(\ham(X,\omega_0))^*\oplus Lie(\ham(X,\omega_0)\times \ham(X,\omega_1))^*\] by 
	\[\mu_{\cJ,0}(J,f):=(\mu_{\cJ},\mu_0).\] We need to show that this is a moment map for the K{\"a}hler form \[\Omega_{\cJ,0}:=\Omega_{\cJ}+\Omega_0\]  For this moment map, the range contains more equations than we want, and the domain $J$ and $f$ has no relation. We will fix this issue by the following steps.\\
	\item[Step 2] Consider the subgroup $H\cong \ham(X,\omega_0)\times \ham(X,\omega_1)$,  and the embedding map 
	$\iota: H\rightarrow  \ham(X,\omega_0)\times \ham(X,\omega_0)\times \ham(X,\omega_1)$ by 
	\[\iota(\sigma, \eta)=(\sigma^{-1},\sigma, \eta)\] It induces a map
	$\iota^*: Lie(\ham(X,\omega_0)\times\ham(X,\omega_0)\times \ham(X,\omega_1))^*\rightarrow Lie(H)^*$,  and the map
	\[\iota^*\circ\mu_{\cJ,0}=\left(\dfrac{\omega_0^{n-1}\wedge(-\Ric(\omega_0,J)+f^*\omega_1-c_1\omega_0)}{(n-1)!}, \dfrac{f_*\omega_0^n-c_2\omega_1^n}{n!}\right)\] is also a moment map.
	\item[Step 3] In order to make sure the solution indeed is K{\"a}hler, we consider the subspace 
	\[\cY_{0}:=\{(J,f)| Df J Df^{-1}\in \cJ(X,\omega_1)\},\] and we need to show that this space has the following properties:
	\begin{enumerate}
		\item It is closed under the action of $H$;
		\item It is a smooth manifold. If we want the solutions to be K{\"a}hler, we need this space to be a K{\"a}hler manifold. 
	\end{enumerate} 
Then $f^*\omega_1$ is $J$ invariant and hence it is a K{\"a}hler form. But our theory also need the domain to be the complexified orbit space $H^{\CC}\cdot (J_0,f_0)$. Notice that this space is equivalent to 
\[\{(\varphi_0, \varphi_1)\in \map(X,X)| \varphi_i^*\omega_i=\omega_i+\ddbar h_i\text{ for some }h_i\in PSH(X,\omega_i), i=0,1\}.\] We will show that the solution of $\iota^*\circ \mu_{\cJ,0}|_{H^{\CC}\cdot (J_0,f_0)}=0$ is equivalent to the solution of ccscK equation in the K{\"a}hler class $[\omega_0], [\omega_1]$. 
\end{enumerate}
\begin{rema}
	\begin{enumerate}
		\item  we may also choose $\Omega_{\cJ,0;a_1,a_2}:=a_1\Omega_{\cJ}+a_2\Omega_0$ in step 1 for some positive number $a_1,a_2$ to affect the constant of the outcome moment map, that is,
		\[\iota^*\circ\mu_{\cJ,0}=\left(\dfrac{\omega_0^{n-1}\wedge(-a_1\Ric(\omega_0,J)+a_2f^*\omega_1+c_1'\omega_0)}{(n-1)!}, \dfrac{a_2f_*\omega_0^n-c_2'\omega_1^n}{n!}\right)\] but $a_1,a_2$ need to be positive so that $\Omega_{\cJ,0}$ is still a symplectic (or K{\"a}hler form if it is $J$ invariant).
		\item  Notice that the embedding is not unique. For example ,it may also be $(\sigma, \sigma ,\eta)$, $(\sigma, \sigma^{-1} ,\eta)$ or $(\sigma^{-1}, \sigma ,\eta^{-1})$ . These embedding change part of the sign of the moment map. For example,  if we change the embedding to be $(\sigma, \sigma ,\eta)$, then the moment map becomes
		\[\iota^*\circ\mu_{\cJ,0}=\left(\dfrac{\omega_0^{n-1}\wedge(-a_1\Ric(\omega_0,J)-a_2f^*\omega_1-c_1'\omega_0)}{(n-1)!}, \dfrac{a_2f_*\omega_0^n-c_2'\omega_1^n}{n!}\right).\]
		\item Notice that if $(J,f)\in \cY_0$, then $(J^{-1},f)=(-J,f)\in \cY_0$. This implies the above choices of embedding won't affect the space $\cY_0$. However, the corresponding Mabuchi functional will be the same.
	\end{enumerate}
\end{rema}
 In general, the main technical part for this set up is to find the correct domain space (which is $\cY_0$ here). We need a space that is closed under the action and is a K{\"a}hler submanifold. In section \ref{sec: Kahler coupled equation p}, we will discuss the difficulties of finding the suitable complex submanifold of $\map(X,X)$ for the moment map $\mu_p$.\\

Similarly, for coupled K{\"a}hler Yang-Mills equation, we first construct \[\mu_{01}=(\mu_{\cJ}+a_1\mu_0^*-a_2\mu_1^*)|_{cKYM_{01}}=\iota^*\circ (\mu_{\cJ},\mu_0^*,\mu_1^*)|_{\cYM_{01}}\] using step 1 and step 2 with a suitable embedding restricted in a suitable subspace $\cYM_{01}^+$. The subspace we take in step 3 should be \[\cYM_{01}\subset \{(J,f,g)\in \cJ(X,\omega_0)\times\map(X_1,X_0;n-2)^+\times \map(X_1,X_0;n-1)^+| g= f^{-1}, Df J Df^{-1}\in \cJ(X,\omega_1)\},\] such that 
\[\Omega_{\cJ,01;\alpha_0,\alpha_1,\alpha_2}=\alpha_0\Omega_{\cJ}-\alpha_1\Omega_0^*+\alpha_2\Omega_1^*>0\}.\] 
Here $\mu^*$ and $\Omega^*$ are defined in Definition \ref{dual moment map}, as we need
\[\int_X (\alpha_1\omega_1\wedge \omega_0^{[n-1]}-\alpha_2\omega_0^{[n]})=0,\] where $\omega^{[k]}=\dfrac{\omega^k}{k!}$. If we take the undual one, $\alpha_0\Omega_{\cJ}-\alpha_1\Omega_0+\alpha_2\Omega_1$ must not be positive.

\subsection{Main result}
To discuss the main result, We first define the coupled p equation. 
\begin{defi}[coupled p equation]\label{eqn: main equation}
	Let $(X,\omega_X) ,(Y,\omega_Y)$ be symplectic manifolds which are diffeomorphic to each other, $0\leq p\leq n-1$ and let $\map(X,Y;p)^+$ be the space of diffeomorphism such that $\omega_X^{n-p}\wedge f^*\omega_Y^p$ is a volume form (see definition (\ref{def: map(X,Y;p)^+})). Then the { couple equation p} is given by \[\mu_p=0,\] where $\mu_p:\map(X,Y;p)^+\rightarrow Lie(\ham(X,\omega_X)\times \ham(Y,\omega_Y))^*$ is defined by
\begin{equation*}
\mu_p(f):=\left(
c_1\dfrac{\omega_X^n}{n!}-\dfrac{\omega_X^{n-p-1}\wedge f^*\omega_Y^{p+1}}{(n-p-1)!(p+1)!}, \quad
c_2\dfrac{\omega_Y^n}{n!}-\dfrac{f_*\omega_X^{n-p}\wedge\omega_Y^{p}}{(n-p)!p!}
\right), 
\end{equation*}
with $c_1,c_2\in \RR$  such that 
\begin{equation*}
\begin{matrix}\displaystyle
\int_X\dfrac{\omega_X^{n-p-1}\wedge f^*\omega_Y^{p+1}}{(n-p-1)!(p+1)!}&=&\displaystyle c_1\int_X\dfrac{\omega_X^n}{n!};\\ \displaystyle
\int_Y\dfrac{f_*\omega_X^{n-p}\wedge\omega_Y^{p}}{(n-p)!p!}&=&\displaystyle c_2\int_Y\dfrac{\omega_Y^n}{n!}.
\end{matrix} 
\end{equation*}
\end{defi}

After that, we will study  the procedure of combining the moment maps $\mu_{p}$ and $\mu_{\cJ}$ by
a special case which we call the generalized ccscK equation:
 \begin{defi}
 	Let $X$ be a compact K{\"a}hler manifold with K{\"a}hler forms $\omega_0,...,\omega_k$. Then we define the generalised ccscK equation to be the following:  
 	\[\left\{\begin{matrix} \displaystyle
 	\sum_{i=0}^k\left(\dfrac{\omega_{i,\varphi_i}^{p_i+1}}{(p_i+1)!}\wedge \dfrac{\omega_{0,\varphi_0}^{n-p_i-1}}{(n-p_i-1)}\right)-\Ric(\omega_{0,\varphi_0}, J_0)\wedge\dfrac{\omega_0^{n-1}}{(n-1)!}-c_0\dfrac{\omega_{0,\varphi_0}^n}{n!}&=&0\\
 	\dfrac{\omega_{0,\varphi_0}^{n-p_1}}{(n-p_1)!}\wedge \dfrac{\omega_{1,\varphi_1}^{p_1}}{p_1!}-c_1\dfrac{\omega_{1,\varphi_1}^n}{n!}&=&0\\
 	\vdots\\
 	\dfrac{\omega_{0,\varphi_0}^{n-p_k}}{(n-p_k)!}\wedge \dfrac{\omega_{k,\varphi_k}^{p_k}}{p_k!}-c_k\dfrac{\omega_{k,\varphi_k}^n}{n!}
 	&=&0.
 	\end{matrix}\right.\]
 \end{defi}
 We will show that this system of equations has a moment map setup. Moreover, there exists a  underlying space which has a K{\"a}hler structure and is compatible with the action. Hence, by considering the orbit space \[\cO_{J,\vec{f}}:=\left(\prod_{i=0}^k\ham_{J_i}^{\CC}(X_i,\omega_i)\right)\cdot (J,f_1,...,f_k).\] Then the moment map equation is given by theorem \ref{thm: ccscK equation gerenal}:
 \begin{theo}
 	Consider the moment map $\mu_{\vec{p}}:\cO_{J,\vec{f}}\rightarrow Lie(H_0\times...\times H_k)^*$ defined by theorem \ref{thm: Kahler moment map of ccscK general} restricted on $\cO_{J,\vec{f}}$. Then $\mu_{\vec{p}}=0$ has a solution  iff the generalized ccscK equation has a solution $(\varphi_0,\cdots,\varphi_k)$.
 \end{theo} 
 
 In particular, if $X_0=\cdots=X_k$, $f_1=f_2=\cdots=f_k=id$, $\vec{p}=(0,...,0)$, then this is the ccscK equation with the classes fixed. Also, using similar idea,  we can get an alternate setup for the coupled K{\"a}hler Yang-Mills equation (see \cite{AGG13}) for the case $G=U(1)^k$.
 
\subsection{Application}

 As a result, similar consequences in \cite{Don00}, \cite{Wang04} (see also \cite{PS2004}, \cite{PS2009}, \cite{Sz10}, \cite{LeSz15}, \cite{AGG13}) can be applied for generalised ccscK:
 \begin{enumerate}
 	\item {Corollary \ref{cor: unique solution}:} The solution is unique up to automorphism.
 	\item Corollary \ref{cor: Aut is reductive}: If the solution of $\mu_{\cJ,p}=0$ exists, then $\displaystyle \bigcap_{i=0}^kAut(X_i,L_i)$ is reductive.
 	\item Corollary \ref{cor: solution exists implies Futaki invariant vanish}: Futaki invariant is given by $\langle\mu_{\cJ,p}, \xi\rangle$; and if solution exists, then Futaki variant are $0$.
 	\item Corollary \ref{cor: extremal metric and ccscK}: Calabi functional is defined by $||\mu_{\cJ,p}||^2$; and if $||\mu||^2$ have a critical point and Futaki invariant vanished implies $\mu=0$ has a solution.	
 	\item  Corollary \ref{cor: unique solution}: The Mabuchi functional can be defined. (See definition \ref{def: Calabi functional and Mabuchi functional}) This functional is geodesic convex along the geodesic $e^{\sqrt{-1}\xi}\cdot p$, where $\xi \in Lie(G)$, and the minimums (if exists) are the solutions of $\mu=0$. 
\item Corollary \ref{unique in toric}: If the manifold $X$ is a toric variety, then the $(S^1)^n$ invariant solution is unique (if it exists).
 \end{enumerate}
\begin{rema}
Notice that if $G=Ham(X,\omega)$, there is no complexified group $G^{\CC}$. We can still define an orbit space, but the uniqueness of the solution still need to investigate. If the orbit is geodesically convex, i.e., any two points can be connected by the geodesic  $e^{\sqrt{-1}\xi}\cdot p$, then the solution is unique. However, in general, by \cite{Dar14}, $Ham^{\CC}(X,\RR)$ is not geodesically convex. Hence the uniqueness still need to study. 
\end{rema}

Denote $\cK(X,\omega_i):=\{h_i\in C^{\infty}(X,\RR)| \omega_{i,h_i}:=\omega_i+\ddbar h_{i}>0\}$. 
The smooth gedosic  we defined is given by 
$(h_{0,t}, h_{1,t},\cdots, h_{k,t})\subset \cK(X,\omega_0)\times \cdots  \cK(X,\omega_k) $ such that for all $0\leq i\leq k$, 
\[\ddot{h}_{i,t}=|\nabla \dot{h}_{i,t}|_{\omega_i}^2.\] 
Consider $k=1$ case. In \cite{AGG13},  denote the space of metic on the line bundle $L$ to be $\cH(L)$, then the geodesic (Proposition 3.17 of \cite{AGG13}) is given by $(h,H)\in \cK(X,\omega)\times \cH(L)$. 
\[\ddot{h}_{0,t}=|\nabla \dot{h}_{0,t}|_{\omega_0}^2, \]
\[\ddot{H}_t-2 d\dot{H}_t(JX_{\dot{h}_{0,t}})+ \sqrt{-1}F_{H_t}(X_{\dot{h}_{0,t}},JX_{\dot{h}_{0,t}}).\] 

Notice that the second equation is twisted by the K\"{a}hler potential $h$, while in our note, the geodesic are independent by each other. 
Therefore, in this note, we show that the functional is convex along different geodesic, which is more natural in the space of $\cK(X_{\omega_0})\times \cdots \cK(X,\omega_k)$.
%The detail will be in section \ref{sec: application}. As a remark, \\
%  
%Besides, we will provide some possible paths to study the  cscK equation. For example, let we choose  $\vec{p}=n-1$, and let $t>0$, then  we can consider a faimily of system of equations:
%\[\begin{matrix}
%(t\omega_{2,h_2(t)})^n-\Ric(\omega_{1,h_1(t)},J)\wedge \omega_{1,h_1(t)}^{n-1}-c_1(t)\omega_{1,h_1(t)}^n&=&0\\
%\omega_{1,h_1(t)}\wedge \omega_{t,h_2(t)}^{n-1}-c_2\omega_{t,h_2(t)}^n&=&0. 
%\end{matrix}\]
%In particular, when $t\rightarrow 0$, we have 
%\[\begin{matrix}
%(S-S_{h_1(0)})\omega_{1,h_1(0)}^n&=&0\\
%\omega_{1,h_1(0)}\wedge \omega_{2,h_2(0)}^{n-1}-c_2\omega_{2,h_2(0)}^n&=&0. 
%\end{matrix}\]
%It may arise a question: suppose $\omega_1$ and $\omega_2$ solves \\ 
%
%Notice that when $t\rightarrow \infty$, then the equation is the coupled equation $\mu_p=0$. So it may also be useful  to discuss the K{\"a}hler structure setting for the coupled equation p:
%\begin{defi}
%Let $X$ and $Y$ as above.	The coupled equation p is given by $\mu_p=0$. That is,
%\[\begin{matrix}
%\dfrac{\omega_X^{n-p-1}\wedge f^*\omega_Y^{p+1}}{(n-p-1)!(p+1)!} &=& c_1\dfrac{\omega_X^n}{n!}\\ \dfrac{f_*\omega_X^{n-p}\wedge\omega_Y^p}{(n-p)!p!} &=& c_2\dfrac{\omega_Y^n}{n!}
%\end{matrix}.\]
%\end{defi} As we mentioned before, we cannot find a suitable complex submanifold of $\map(X,Y;p)^+$ for the moment map equation $\mu_p=0$. However, some partial result will be provided.

\subsection{More result of $\mu_p$}
After the above applications, we will study if the couple p equation can be viewed as a moment map in a  K\"{a}hler manifold. Unfortunately, there is no K\"{a}hler submanifold in the domain which is closed under the action of $\ham(X,\omega_X)\times \ham(Y,\omega)$, and $\Omega$ is non -degenerated . The best result is  in Proposition \ref{prop:  pseudo moment map}, which implies that $\mu_p$ is a pseudo moment map in $\cX_{f,p}^+$ (as $\Omega_p|_{\cX_{f,p}^+}$ may be degenerated).

 After that, we give a special case for the moment map $\mu_p$ is still a moment map when $X$ is a submanifold of $Y$. 

Finally, in the appendix, we will give a rough idea about the relation between this setup and the setup in \cite{AGG13} and \cite{DaPi19}.

  \begin{ack}  First, the author would express the appreciation from the Asian Journal of Mathematic, and the comment from the reviewer, which points out some mistakes in the earlier version. The author thank for my advisors in my doctorate degree, Jacob Sturm and Prof. Xiaowei Wang, for many advises.  Besides, I appreciate the help from Raymond Yat Tin Chow, and the discusion with Ved V. Datar, K.K. Kwong, Alen Man Chun Lee, John Man Shun Ma, Macro Yat Hin Suen and YingYing Zhang. The last but not the least, I appreciate the help from my former advisor, Changzheng Li and I want to say thank you for the explanation about the moment map theory in cKYM equation by my current advisor Mario Garcia-Fernandez.

Finally, this work is supported from Sun Yat-Sen University and the author is currently partially funded by Grant EUR2020-112265 funded by MCIN/AEI/10.13039/501100011033 and by 
the European Union NextGenerationEU/PRTR, and Grant CEX2019-000904-S funded by MCIN/AEI/10.13039/501100011033, and also funded by MICINN under grant PID2019-109339GA-C32.
  \end{ack}
\section{moment map for coupled equation p}
In this section, we will define a class of moment map $\mu_p$ on a open subset of $\map(X,Y;p)^+\subset \map(X,Y)$, with a sympectic form $\Omega_p$ on $\map(X,Y;p)^+$.
To do so, first, we define the domain $\map(X,Y;p)^+$ and the symplectic form $\Omega_p$ on $\map(X,Y;p)^+$:
\begin{defi}\label{def: map(X,Y;p)^+}
Let $(X,\omega_X),(Y,\omega_Y)$ be two compact symplectic manifolds which are diffeomorphic to each other. We define $(\map(X,Y;p)^+,\Omega_p)$ to be the space 
 \[\map(X,Y;p)^+:=\{f\in\diff(X,Y)| \omega_X^{n-p}\wedge f^*\omega_Y^p>0\}.\] that is, $\omega_X^{n-p}\wedge f^*\omega_Y^p$ is a volume form,  with the symplectic form 
\begin{align*}\Omega_p(\delta_1f,\delta_2f):=&\dfrac{1}{(n-p)!p!}\int_X \omega_Y(\delta_1f,\delta_2f) \omega_X^{n-p}\wedge f^*\omega_Y^p,
\end{align*} 
where $\delta_1f,\delta_2f\in T_f(\map(X,Y))=f^*(TY):=\{s_f: X \rightarrow TY|| s_f|_x\in T_{f(x)}Y\} $.
\end{defi}	

Notice that $\Omega_p$ is a symplectic form on $\map(X,Y;p)$ as $\omega_Y$ is non degenerate and closed. Also, we have a group $H:=\ham(X,\omega_X)\times\ham(Y,\omega_Y)$ acts on $\map(X,Y;p)$ defined by 
\[(\sigma,\eta)\cdot f:=\eta\circ f\circ \sigma^{-1},\] where $\ham(X,\omega_X)$ and $\ham(Y,\omega_Y)$ are the Hamiltonian groups with respect to $\omega_X$ and $\omega_Y$ respectively. Also, $\map(X,Y;p)^+$ is an open set in $\map(X,Y;p)$, hence it is also a symplectic manifold.

We will first show that the Hamiltonian action on $\map(X,Y;p)$ is closed in $\map(X,Y;p)^+$. Then up to constants, we can define a map $\mu_p$ by 
\begin{defi}\label{eqn: main moment map}
 We denote the map
 \[\mu_p:\map(X,Y;p)^+ \rightarrow Lie(\ham(X,\omega_X)\times \ham(Y,\omega_Y))^*\] by
\begin{equation}
\mu_{p,\omega_X\omega_Y}(f)=\dfrac{n}{n-p}\left(c_1\dfrac{\omega_X^n}{n!}-\dfrac{\omega_X^{n-p-1}\wedge f^*\omega_Y^{p+1}}{(n-p-1)!(p+1)!}, \dfrac{f_*\omega_X^{n-p}\wedge\omega_Y^p}{(n-p)!p!}-c_2\dfrac{\omega_Y^n}{n!}\right),
\end{equation}
where $\displaystyle c_1=\dfrac{\int_X\omega_X^{n-p-1}\wedge f^*\omega_Y^{p+1}}{\int_X\omega_X^n}$, $\displaystyle c_2=\dfrac{\int_Yf_*\omega_X^{n-p}\wedge\omega_Y^p}{\int_Y\omega_Y^n}$. 
\end{defi}
We will show that this is a moment map corresponding to  $(\map(X,Y;p)^+,\Omega_p)$.
In particular, if we take $(X,\omega_X)=(X,\omega_0)$ and $(Y,\omega_Y)=(X,\omega_1)$, then we will get the moment map for coupled equation p. \\

\begin{lemm}\label{lem: closedness of positivity}
The group action $H$ on $\map(X,Y;p)$ is closed in $\map(X,Y;p)^+$. Also, $\Omega_p$ is invariant under the action of $H$ for $p=0,...,n-1$. 
\end{lemm}
\begin{proof}
Let $\varphi:X\rightarrow \RR$ be a test function, that is $\varphi\geq 0$ is a smooth function, and there exists $x\in X$ such that $\varphi(x)>0$

Let $f\in \map(X,Y;p)^+$ and denote $u=\sigma^{-1}(x)$. Then 
 \begin{align*}
 \int_X \varphi(x) \omega_X^{n-p}\wedge (\eta\circ f\circ \sigma^{-1})^*\omega_Y^p=&\int_X \varphi(x) \omega_X^{n-p}\wedge (\sigma^{-1})^* f^*\eta^*\omega_Y^p\\
 =&\int_X \varphi(u) \sigma^*\omega_X^{n-p}\wedge  f^*\eta^*\omega_Y^p|_u\\
 =&\int_X\varphi(u)\omega_X^{n-p}\wedge f^*\omega_Y^p|_u\\
 =&\int_X\varphi(x)\omega_X^{n-p}\wedge f^*\omega_Y^p>0.
\end{align*}
Hence $\eta\circ f\circ\sigma^{-1}\in \map(X,Y;p)^+$. Also,if we choose $\varphi$ such that $\varphi(x)=1$, by the above calculation, we can see that the volume is unchanged. Finally,
 notice that $(\sigma,\eta)\cdot (\delta f)|_x=9D\eta) (\delta f)|_{\sigma^{-1}(x)}$. Hence
\begin{align*}
&\int_{X}\omega_Y(D\eta)(\delta_1 f), (D\eta)(\delta_2 f))|_{\eta\circ f\circ\sigma^{-1}(x)} \omega_X^{n-p}\wedge (\sigma^{-1})^*f^*\eta^*\omega_Y^{p}\\
=&\int_{X}\omega_Y((D\eta)(\delta_1 f),(D\eta)(\delta_2f))|_{\eta\circ f(x)} \sigma^*\omega_X^{n-p}\wedge f^*\sigma_Y^*\omega_Y^{p}\\
=&\int_{X}\eta^*\omega_Y((\delta_1 f), (\delta_2 f))|_{ f(x)} \sigma^*\omega_X^{n-p}\wedge f^*\eta^*\omega_Y^{p}\\
=&\int_{X}\omega_Y((\delta_1 f),(\delta_2 f))|_{ f(x)} \omega_X^{n-p}\wedge f^*\omega_Y^{p}.
\end{align*}
\end{proof}
\begin{rema}
The proof also applies to the case $p=n$. Indeed, $\sym(X,\omega_X)\times\sym(Y,\omega_Y)\subset \sym(\map(X,Y), \Omega_p)$ for all $p=0,...,n$.
\end{rema}
%Now,  by the volume form we defined, for all $(\xi_X,\xi_Y), (\xi_X',\xi_Y')\in \ham(X)\times \ham(Y)$, The inner product is defined by 
%\[\langle(\xi_X,\xi_Y), (\xi_X',\xi_Y')\rangle:=\int_Xf^*\omega_Y\]
Before we prove the first main theorem, we first prove a technical lemma.
\begin{lemm}\label{lem: weak interior product commute}
	Let $(X,\alpha,\beta)$ be a symplectic manifold with symplectic forms $\alpha,\beta$. Denote $\gamma_p:=\alpha^{n-1-p}\wedge \beta^p$. If $\alpha^{n-p}\wedge \beta^p>0$, then for any $u,v\in TX$,
	\[n\iota_u\alpha \wedge \iota_v\beta \wedge \gamma_p=-\beta(u,v) \alpha\wedge\gamma_p.\]
Similarly, we have
\[n\iota_u\alpha \wedge \iota_v\alpha \wedge \gamma_p=-\alpha(u,v) \alpha\wedge\gamma_p.\]
\end{lemm}
\begin{proof}
	We prove it on local coordinate. As $\alpha\wedge\gamma_p>0$, it is a volume form, so if we denote $\alpha=A_{ij}$, then for any 2 form $\eta=d_{ij}$, 
	\[\dfrac{n\eta\wedge \gamma_p}{\alpha\wedge\gamma_p}=d_{ij}A^{ji},\] where $A^{ji}$ is the inverse matric of $A_{ij}$. As a result, if we denote $u=u^i$, $v=v^i$, $\beta=B_{ij}$, then 
	\[\dfrac{n\iota_u\alpha \wedge \iota_v\beta \wedge \gamma_p}{\alpha\wedge\gamma_p}=u^iA_{ij}v^kB_{kl}A^{lj}=u^lv^kB_{kl}=-\beta(u,v).\]
the second statement follows from the same proof.
\end{proof}
Notice that 
We now give the first theorem of this note.
\begin{theo}\label{thm: main theorem of moment map}
	Let $0\leq p\leq n-1$, then the map $\mu:\map(X,Y;p)^+\rightarrow Lie(\ham(X,\omega_X)\times \ham(Y,\omega_Y))^*$ defined by \[\mu_p(f):=\dfrac{n}{n-p}\left(c_1\dfrac{\omega_X^n}{n!}-\dfrac{\omega_X^{n-1-p}}{(n-1-p)!}\wedge f^*\dfrac{\omega_Y^{p+1}}{(p+1)!},c_2\dfrac{\omega_Y^n}{n!}-\dfrac{f_*\omega_X^{n-p}}{(n-p)!}\wedge \dfrac{\omega_Y^p}{p!}\right)\] is a moment map with respect to the action $\eta\circ f\circ\sigma^{-1}$, where 
	\[c_1:=c_1(f)=\dfrac{n!}{(n-1-p)!(p+1)!}\dfrac{\int_X\omega_X^{n-1-p}\wedge f^*\omega_Y^{p+1}}{\int_X\omega_1^n}, c_2:=c_2(f)=\dfrac{n!}{(n-p)!p!}\dfrac{\int_Yf_*\omega_X^{n-p}\wedge \omega_Y^p}{\int_Y\omega_Y^n}.\] 
\end{theo}
\begin{proof}
	Recall that $Lie(\ham(X,\omega_X))\cong C_0^{\infty}(X, \RR):=\{\varphi| \int_X\varphi\omega_X^n=0\}$ such that for any $\varphi\in C^{\infty}(X,\RR)$,
	\[d\varphi=\iota_{\xi_{\varphi}}\omega_X.\]
	Hence we have $Lie(\ham(X,\omega_X))^*\cong \Omega^n(X,\RR)$, the space of volume form of $X$.  \\

	We let $(\varphi,\psi)\in C^{\infty}(X)\times C^{\infty}(Y)$ and $H_{(\varphi,\psi)}(f):=\langle\mu_p(f),(\varphi,\psi)\rangle$. Then
	\[H_{(\varphi,\psi)}(f)=\dfrac{n}{n-p}\left(-\int_X\varphi \dfrac{\omega_X^{n-1-p}}{(n-1-p)!}\wedge f^*\dfrac{\omega_Y^{p+1}}{(p+1)!}+c_1\int_X\varphi\dfrac{\omega_X^n}{n!}-\int_Y \psi \dfrac{f_*\omega_X^{n-p}}{(n-p)!}\wedge \dfrac{\omega_Y^p}{p!}+c_2\int_Y\psi\dfrac{\omega_Y^n}{n!}\right).\]
Our goal is to show that 
\[\iota_{X_{(\varphi,\psi)}}\Omega_p(v)=dH_{(\varphi,\psi)}(v).\]

  Let $f_t$ be a family of diffeomorphism. By defining $\eta_t=f_t\circ f^{-1}$, then 
	$f_t=\eta_t\circ f$.  we denote \[v=\left.\dfrac{d}{dt}\right|_{t=0}\eta_t \in T_{id}(\map(Y,Y;p)^+). \] 	Notice that $f^*\omega_Y$ is a  symplectic form which is closed,
	therefore, 
	\[\left.\dfrac{d}{dt}\right|_{t=0}f_t^*\omega_Y=\left.\dfrac{d}{dt}\right|_{t=0}f^* \eta_t^*\omega_Y=f^*\mathcal{L}_v\omega_Y=f^*d\iota_v\omega_Y=df^*\iota_v\omega_Y.\]
	Similarly, 
	\[\left.\dfrac{d}{dt}\right|_{t=0}{f_t}_*\omega_X=\left.\dfrac{d}{dt}\right|_{t=0}{\eta_t}_*{f}_*\omega_X=d\iota_v f_*\omega_X.\]  
	Notice that these are exact forms. By the fact that for any compact manifold $M^{2n}$, and for any $2k-1$ form $\beta$ and $2n-2k$ closed form $\alpha$,
		\[\int_M \alpha \wedge d\beta=-\int_M (d \alpha)\wedge \beta=0.\] 
		It implies
		\[\dfrac{d}{dt}\int_X \omega_X^{n-p}\wedge f_t^*\omega_Y^p=\dfrac{d}{dt}\int_Y {f_t}_*\omega_X^{n-p}\wedge \omega_Y^p=0.\] 
	We now identify  $ v'\in T \map(X,Y)$ and $v\in T\map(Y,Y)$ by 
	\[v|_{f(x)}=v'|_x.\] Then 
	\begin{align*}
	&dH_{(\varphi,\psi)}(v')\\
	=&\dfrac{-n}{(n-p)!(p)!}\int_X\varphi \omega_X^{n-1-p}\wedge f^*\omega_Y^p\wedge \left(\dfrac{d}{dt}f_t^*\omega_Y(\cdot, \cdot)\right)-\dfrac{n}{(n-p)!p!}\int_Y \psi \left(\dfrac{d}{dt}(f_t^{-1})^*\omega_X\right)\wedge f_*\omega_X^{n-p-1}\wedge \omega_Y^p\\
	=&\dfrac{-n}{(n-p)!(p)!}\int_X\varphi\omega_X^{n-1-p}\wedge f^*\omega_Y^p\wedge df^*\iota_{v}\omega_Y+\dfrac{n}{(n-p)!p!}\int_Y \psi d\iota_{v}f_*\omega_X\wedge f_*\omega_X^{n-p-1}\wedge \omega_Y^p\\
	=&\dfrac{n}{(n-p)!(p)!}\int_Xd\varphi\wedge f^*\iota_{v}\omega_Y\wedge\omega_X^{n-1-p}\wedge f^*\omega_Y^p-\dfrac{n}{(n-p)!p!}\int_Y d\psi \wedge \iota_{v}f_*\omega_X\wedge f_*\omega_X^{n-p-1}\wedge \omega_Y^p\\
	=& \dfrac{n}{(n-p)!(p)!}\int_X \iota_{\xi_{\varphi}}\omega_X\wedge f^*\iota_v\omega_Y \wedge \alpha_f-\dfrac{n}{(n-p)!p!}\int_Y \iota_{\xi_{\psi}}\omega_Y\wedge \iota_vf_*\omega_X\wedge f_*\alpha_f,
	\end{align*}
	where $\alpha_f=\omega_X^{n-p-1}\wedge f^*\omega_Y^{p}$. 
	By Lemma \ref{lem: weak interior product commute}, as $f_*\alpha_f=f_*\omega_X^{n-p-1}\wedge \omega_Y^{p}$,
	\[n\iota_{\xi_{\psi}}\omega_Y\wedge \iota_vf_*\omega_X\wedge \alpha_f=-n\iota_vf_*\omega_X\wedge\iota_{\xi_{\psi}}\omega_Y\wedge \alpha_f=\omega_Y(v,\xi_{\psi})f_*\omega_X\wedge\alpha_f=-\omega_Y(\xi_{\psi},v)f_*\omega_X\wedge\alpha_f.\]
 Moreover,
	\begin{align*}
	n\iota_{\xi_{\varphi}}\omega_X\wedge f^*\iota_v\omega_Y\wedge\alpha_f|_x
	=&nf^*\left(\iota_{f_*\xi_{\varphi}}f_*\omega_X\wedge\iota_v\omega_Y\wedge f_*\alpha_f|_{f(x)}\right)\\
	=&-\omega_Y(f_*\xi_{\varphi}|_x,v|_{f(x)})f^*( f_*\omega_X\wedge f_*\alpha_f|_{f(x)})\\
	=&-\omega_Y(f_*\xi_{\varphi},v') \omega_X\wedge\alpha_f.
	\end{align*}
	Therefore,
	\begin{align*}
	&dH_{(\varphi,\psi)}(v')\\
	=& \dfrac{1}{(n-p)!p!}\left(-\int_X \omega_Y(f_*\xi_{\varphi},v') \omega_X \wedge \alpha_f+\int_Y \omega_Y(\xi_{\psi},v)f_*\omega_X\wedge f_*\alpha_f\right)\\
	=&\dfrac{1}{(n-p)!p!}\left(-\int_X\omega_Y(f_*\xi_{\varphi},v')\omega_X \wedge \alpha_f+\int_X \omega_Y(\xi_{\psi}\circ f,v')\omega_X\wedge \alpha_f\right)\\
	=&\dfrac{1}{(n-p)!p!}\left(\int_X \omega_Y(\xi_{\psi}\circ f,v')\omega_X \wedge \alpha_f-\int_X\omega_Y(f_*\xi_{\varphi},v')\omega_X \wedge \alpha_f\right).
	\end{align*}
	On the other hand, for the action $\eta\circ f\circ\sigma^{-1}$ with $(\varphi,\psi)\in C^{\infty}(X)\times C^{\infty}(Y)$, the induced vector field is given by 
	\[X_{(\varphi,\psi)}= \xi_{\psi}\circ f-f_*\cdot \xi_{\varphi}.\]  So 
	\[\iota_{X_{(\varphi,\psi)}}\Omega_p(v)=\dfrac{1}{(n-p)!p!}\int_X\omega_Y(\xi_{\psi}\circ f-f_*\cdot \xi_{\varphi},v')\omega_X \wedge \alpha_f=dH_{(\varphi,\psi)}(v).\]
\end{proof}
\begin{rema}
Notice that $\map(X,Y;p)^+$ is an open set in $\map(X,Y;p)$, hence it is still a symplectic submanifold. Also, $c_1$, $c_2$ may not be constant, as $\diff(X,Y;p)^+$ may not be connected. However, if $f_1,f_2$ are path connected, then $c_1(f_1)=c_1(f_2)$ and $c_2(f_1)=c_2(f_2)$.
\end{rema}
\begin{defi}
We call the above moment map to be the moment map $p$ with respect to $\omega_X,\omega_Y$, denoted as $\mu_{p;\omega_X,\omega_Y}$, or simply $\mu_p$ if no confusion arises.
\end{defi}

Finally, we define the "dual" moment map by the following.
\begin{defi}\label{dual moment map}
	We define the dual moment map of $\mu_p$ to be $\mu_{p}^*:\map(Y,X;n-p-1)^+\rightarrow \ham(Y,\omega_Y)\times \ham(X,\omega_X)$, with \[\mu_{p }^*(g):=\mu_{n-p-1,\omega_Y,\omega_X}(g)=\dfrac{n}{p+1}\left(c_1\dfrac{\omega_Y^n}{n!}-\dfrac{\omega_Y^{p}\wedge g^*\omega_X^{n-p}}{p!(n-p)!}, c_2\dfrac{\omega_X^n}{n!}-\dfrac{g_*\omega_Y^{p+1}\wedge\omega_X^{n-p-1}}{(p+1)!(n-p-1)!}\right).\]
\end{defi}
Notice that $(\mu_{p}^*)^*=\mu_p$. Also, it is obvious that
\begin{lemm}\label{dual moment map lemma}
	$f:X\rightarrow Y$ solves the coupled equation $p$ (i.e., $\mu_p(f)=0$) iff $f^{-1}$ solve 
	\[\mu_{p}^*(g)=0.\]
\end{lemm}
\begin{proof}
	$f^*=f_*^{-1}$ and $f_*=(f^{-1})^*$; and the result follows.
\end{proof}
%We can also define $\cZ_p^*=\cZ_{n-p-1}\subset \map(Y,X)$. Then we can consider \[\cY_{p}^*:=\{(J,g)\in \cJ(X,\omega_X)\times \map(Y,X)| g_*J= Dg^{-1}JDg\in \cJ(Y,\omega_Y)\}\subset \cJ(X,\omega_X)\times \cZ_p^*.\] 
%Then $\ham(Y,\omega_Y)\times \ham(X,\omega)$ action is closed in $\cY_p^*$, and we will basically get the same moment map equation.

 The main difference between $\mu_p(f)$ and $\mu_p^*(f^{-1})$ is the following: if we put $g=f^{-1}$, and we reorder the domain into $Lie (\ham(X,\omega_X)\times\ham(Y,\omega_Y))^*$,
\[\mu_p^*(f^{-1})=\dfrac{p+1}{n-p}\mu_p(f).\] Hence, we can change the sign of the moment map without changing the action on $\map(X,Y)$. 

Also, we can change the sign by changing the action. For example, we may change the action to be $\eta^{-1}\circ f \circ\sigma^{-1}$, then the sign of the second part of the moment map will change.

\section{moment map picture for coupled equations with curvature}\label{sec: ccscK}

\subsection{Combining moment maps}
We now use this moment map to get some coupled equations related to Ricci curvature.   Recall that we have the following fact: 

\begin{lemm}
	Let $(M_1,\alpha_1)$, $(M_2,\alpha_2)$ be two symplectic manifolds with hamiltonian group action $G_1,G_2$, and let their corresponding moment map be $\mu_i:M_i\rightarrow Lie(G_i)^*$. Let $H$ be a subgroup of $G_1\times G_2$ and $M$ be a (even dimensional) submanifold of $M_1\times M_2$ such that $H$ is closed under $M$ and $(\Omega_1+\Omega_2)|_M$ is non degenerate (i.e, it is a symplectic form). Then the map 
	$\mu:M\rightarrow Lie(H)^*$ corresponding to the symplectic form $\Omega_1+\Omega_2$  defined by 
	\[\mu_H=\proj_{Lie(H)^*}(\mu_1,\mu_2)|_{M}\] is a moment map, where $\proj_{Lie(H)^*}:Lie(G_1)^*\times Lie(G_2)^*\rightarrow Lie(H)^*$ is the projection map. 
\end{lemm}
Using this lemma, we can combine the moment map we defined above and the scalar curvature to get different coupled equations.
\begin{rema}
If $\Omega$ is a K{\"a}hler form, and $M$ is a complex submanifold, then $\Omega|_M$ is also a K{\"a}hler form. However, in general, for a submanifold $M$, $\Omega|_M$ may be degenerate. For example, we may take $M\subset L$, where $L$ is the Lagrangian of $M_1\times M_2$. 
\end{rema}

	 Let $Y=X$ with symplectic forms $\omega_i$. 
Let $\cZ_i:=\map((X,\omega_0), (X,\omega_i);p)^+$, \[\Omega_{p,i}((\delta f)_1,(\delta f)_2):=\dfrac{1}{(n-p)!p!}\int_X \omega_{i}((\delta f)_1,(\delta f)_2)\omega_0^{n-p}\wedge f^*\omega_i^p.\] Denote $\ham(X,\omega_i):=H_i$, and we denote the corresponding moment map to be $\mu_{p,i}$, in which 
	\[\mu_{p,i}(f):=\dfrac{n}{n-p}\left(c_{1,i}\dfrac{\omega_0^n}{n!}-\dfrac{\omega_0^{n-1-p}}{(n-1-p)!}\wedge f^*\dfrac{\omega_i^{p+1}}{(p+1)!},\dfrac{f_*\omega_0^{n-p}}{(n-p)!}\wedge \dfrac{\omega_i^p}{p!}-c_{2,i}\dfrac{\omega_i^n}{n!}\right),\]  by theorem \ref{thm: main theorem of moment map}. Then by considering the space $\cZ_1\times \cZ_3\times \cdots \cZ_k$ with 
	\[\Omega=\sum_{i=1}^k\dfrac{n-p}{n}\pi_{i-1}^*\Omega_{p,i},\] we have a moment map 
	$\displaystyle\overline{\mu_p}:\cZ_1\times...\cZ_k\rightarrow \prod_{i=1}^k(H_0\times H_i)$ defined by 
	\[\overline{\mu_p}=\dfrac{n-p}{n}(\mu_{p,1},...,\mu_{p,k}).\]
	
	The next step is finding the suitable subgroup so that the image of moment map can be combined. To be precise,   the embedding $\iota:H_1\times...H_k\rightarrow (H_0\times H_1)\times\cdots\times (H_0\times H_k)$ defined by 
	\[\iota(\sigma_0,...,\sigma_k)=(\sigma_0,\sigma_1,\sigma_0,\sigma_2,...,\sigma_0,\sigma_k)\] induces a map $\displaystyle\iota:Lie(\prod_{i=1}^k(H_0\times H_i))^*\rightarrow Lie(H_0\times H_2\times ...\times H_k)^*$.  Hence the moment map 
	$\mu_p:=\iota^*\circ\overline{\mu_p|_{\cZ}}$ is given by 
	\[\mu_p(f_1,...,f_k)=\left\{\begin{matrix}
	c_0\dfrac{\omega_0^n}{n!}-\dfrac{f_1^*\omega_1^{p+1} +...+f_k^*\omega_k^{p+1}}{(p+1)!}\wedge \dfrac{\omega_0^{n-p-1}}{(n-p-1)!}\\
	\dfrac{{f_1}_*\omega_0^{n-p}}{(n-p)!}\wedge \dfrac{\omega_1^p}{p!}-c_1\dfrac{\omega_1^n}{n!}\\
	\vdots\\
	\dfrac{{f_k}_*\omega_0^{n-p}}{(n-p)!}\wedge \dfrac{\omega_k^p}{p!}-c_k\dfrac{\omega_k^n}{n!}.
	\end{matrix}\right.\]

In general, for different $i$, we can choose different $0\leq p_i\leq n-1$, hence we have the following:
\begin{lemm}\label{lem: moment map of twisted equation}
	Let $(X_i,\omega_i,J_i)$ be K{\"a}hler manifolds, and $X_0$ is diffeomorphic to $X_i$ for all $i=0,1,...,k$. Denote $\vec{p}=(p_1,...,p_k)$. Consider the space 
	\[\cZ_{\vec{p}}:=\prod_{i=1}^k \map(X_0,X_i)_{p_i}^+,\] where 
	\[\map(X_0,X_i)_p^+:=\{f\in\map(X_0,X_i)| \omega_0^{n-p_i}\wedge f^*\omega_i^{p_i}>0\}.\]  We define the symplectic form   on $\cZ(p_1,...,p_k)$ by
	\[\Omega_{\vec{p}}((v_1,...,v_k),( w_1,...,w_k)):=\sum_i\dfrac{n-p_i}{n}\int_X\omega_i(v_i,w_i)\omega_0^{n-p_i}\wedge \omega_i^{p_i}.\] Then with the action of $\displaystyle\prod_{i=0}^k\ham(X_i,\omega_i)$, the moment map is given by 
	\[\mu_{\vec{p}}(\vec{f}):=\left\{\begin{matrix} \displaystyle
	c_0\dfrac{\omega_0^n}{n!}-\sum_{i=1}^k\dfrac{\omega_0^{n-p_i-1}}{(n-p_i-1)!}\wedge \dfrac{f_i^*\omega_i^{p_i+1}}{(p_i+1)!}\\
	\dfrac{{f_1}_*\omega_0^{n-p_1}}{(n-p_1)!}\wedge \dfrac{\omega_1^{p_1}}{p_1!}-c_1\dfrac{\omega_1^n}{n!}\\
	\vdots\\
	\dfrac{{f_k}_*\omega_0^{n-p_k}}{(n-p_k)!}\wedge \dfrac{\omega_k^{p_k}}{p_k!}-c_k\dfrac{\omega_k^n}{n!}
	\end{matrix}\right..\]
	
	Also, by identifying $\cZ_i$ and $\cZ_i^*$, and considering 
	\[\Omega^*=\sum_{i=1}^k\dfrac{p+1}{n}\pi_{i-1}^*|\Omega_i,\] we have 
		\[\mu^*_{\vec{p}}(\vec{f}):=\left\{\begin{matrix} \displaystyle
	\sum_{i=1}^k\dfrac{\omega_0^{n-p_i-1}}{(n-p_i-1)!}\wedge \dfrac{f_i^*\omega_i^{p_i+1}}{(p_i+1)!}-c_0\dfrac{\omega_0^n}{n!}\\
	c_1\dfrac{\omega_1^n}{n!}-\dfrac{{f_1}_*\omega_0^{n-p_1}}{(n-p_1)!}\wedge \dfrac{\omega_1^{p_1}}{p_1!}\\
	\vdots\\
	c_k\dfrac{\omega_k^n}{n!}-\dfrac{{f_k}_*\omega_0^{n-p_k}}{(n-p_k)!}\wedge \dfrac{\omega_k^{p_k}}{p_k!}
	\end{matrix}\right..\]
\end{lemm}

Let $(\cJ(X,\omega_X),\Omega_J)$ to be the space of all integrable almost complex structure which are compatible to $\omega_X$, and for all $A,B\in T_J\cJ(X,\omega_X)$,
\[\Omega_J(A,B)=\dfrac{1}{n!}\int_X \langle A,B\rangle_{g_J}\omega_X^n,\] where $g_J(v,w)=\omega(v,Jw)$. Also, let the action $\ham(X,\omega_X)$ acts on $\cJ(X,\omega_X)$ by 
\[\sigma\cdot J=D\sigma^{-1}\cdot J\cdot D\sigma,\] and denote  
\[\wedge_0^{n}(X):=\{\alpha\in \wedge^n(X)| \int_X \alpha=0\}.\]Then we have a moment map (\cite{Don00},\cite{Don01}) \[\mu_J:\cJ(X,\omega_X)\rightarrow Lie(\ham(X,\omega_X))^*\cong \wedge_0^{n}(X)\] which is given by 
\[\mu_J(\sigma_{\varphi})=\Ric(\omega_{\varphi})\wedge\dfrac{\omega_{
	\varphi}^{n-1}}{(n-1)!}-\bar{S}\dfrac{\omega^n}{n!}=\left(S_{\varphi}-\overline{S}\right)\omega_{X,\varphi}^n,\] where $\sigma_{\varphi}^*\omega_X=\omega_{X,\varphi}$. To sum up, we have the following lemma.
\begin{lemm}\label{lem: moment map of ccscK 1}
Let $\cZ_{\vec{p}}$ be as above and  consider $\cJ(X,\omega_X)\times \cZ_{\vec{p}}, \Omega_{J,{\vec{p}}}:=\pi_{\cJ}^*\Omega_J+\pi_{\cZ}^*\Omega_{\vec{p}}$, then we have a moment map \[\widehat{\mu}_{\cJ,\vec{p}}:\cJ(X,\omega_X)\times \cZ_{\vec{p}} \rightarrow Lie(H_1\times H_1\times ...\times H_k)^*\] defined by 
\[\widehat{\mu}_{\cJ,\vec{p}}(J,f_1,...,f_k)=(\mu_J,\mu_{\cZ}^*).\] Moreover, by considering the group action $\iota:  H_0\times...\times H_k\rightarrow H_0\times H_0\times H_1\times...\times H_k$  by 
\[\iota(\sigma_0,\sigma_1,...,\sigma_k)=(\sigma_0^{-1},\sigma_0,\sigma_1,...,\sigma_k)\]we can restrict the moment map to be
\[{\mu}_{\cJ,\vec{p}}:\cJ(X,\omega_X)\times \cZ_{\vec{p}} \rightarrow Lie( H_0\times ...\times H_k)^*\] which is given by
\[{\mu}_{\cJ,\vec{p}}(J,f_1,...,f_k)=\begin{pmatrix} \displaystyle c_0\dfrac{\omega_0^n}{n!}-
\sum_{i=0}^k\left(\dfrac{f_i^*\omega_i^{p_i+1}}{(p_i+1)!}\wedge \dfrac{\omega_0^{n-p_i-1}}{(n-p_i-1)!}\right)+\Ric(\omega_0, J)\wedge\dfrac{\omega_0^{n-1}}{(n-1)!}\\
c_1\dfrac{\omega_1^n}{n!}-\dfrac{(f_1)_*\omega_0^{n-p_1}}{(n-p_1)!}\wedge \dfrac{\omega_1^{p_1}}{p_1!}\\
\vdots\\
c_k\dfrac{\omega_k^n}{n!}-\dfrac{(f_k)_*\omega_0^{n-p_k}}{(n-p_k)!}\wedge \dfrac{\omega_k^{p_k}}{p_k!}
\end{pmatrix}.\]
\end{lemm}
\begin{rema}
	We can consider the action on $\cZ_i$ to be $(\sigma,\eta)\cdot f_i=(\eta^{-1}\circ f_i \circ \sigma^{-1})$, then we can change the sign of all the expression $c_i\omega_0^n-{f_i}_*\omega_i^n$.
\end{rema}
Notice that it is not the equation we aim to obtain yet.  In the next section, we will define a suitable submanifold as the domain of the moment map, and discuss how to transform this moment map equation into the moment map equation we want.  

\subsection{K{\"a}hler structure on generalized ccscK}
We now define the domain of the generalized ccscK $\cY_{\vec{p}}$, which hope to be the largest K\"{a}hler manifold which is closed in the group action, and  \[\cY_{\vec{p}}\subset \{(\omega_0,\cdots \omega_k)\in \Omega^2(X_0, \RR)\times \cdots \times\Omega^2(X_k,\RR)| \omega_i \text{ is K\"{a}hler}\}.\]
This space is important as it is useful to study the deformation of solutions. Also, with this K\"{a}hler manifold, any complex orbit is a K\"{a}hler manifold. 
\begin{defi}\label{def: Y}
	Denote $J^f:= Df J Df^{-1}$. Define $\cY_{\vec{p}}\subset\cJ(X_0,\omega_0)\times \cZ_{\vec{p}}$ by 
	\[\cY_{\vec{p}}:=\{(J,f_1,..,f_k)| J^{f_i} \in \cJ(X_i,\omega_i).\}\]
\end{defi}

Our goal is to show that $\cY_{\vec{p}}$ is K\"{a}hler with respect to the symplectic form \[\Omega_{\cJ,p}:=\Omega_J+\Omega_{\vec{p}}.\] 
%\begin{defi}
%A manifold $Y$ is a smooth manifold if at any point, $Y$ is loally homemorphic to its tagnent space, and the change of coordinate map is smooth.
%A complex manifold with respect to a almost complex structure $J$ if the transition map preserve $J$, that is, if $F_{\alpha\beta}$ is the transition map, then
%\[(J|_{F_{\alpha\beta}(p)})\circ (DF|_p (v))= DF_{\alpha\beta}|_{p} (J|_pv) \]  A couple $(Y,\Omega, J)$ is K\"{a}hler if $(Y,J)$ is a complex manifold, $\Omega$ is a non degenerated closed 2 form such that 
%\[\Omega(J \bullet, J\bullet)=\Omega(\bullet, \bullet).\] 
%\end{defi}

As a remark, in \cite{DaSm02}, the defintion of complex manifold is really the classical one; locally homeomorphic to the tangent space, and the change of coordinate maps is biholomorphic. Or in this case, the change of coordinate maps perserve the $J$. 

	Notice that we have a natural almost complex structure on $\displaystyle\cJ(X_0,\omega_0)\times \prod_{i=1}^k\map(X_0,X_i)$, denote by $\hat{J}$, which
	\[\hat{J}(\delta J,\delta f_1,...,\delta f_k)=(J\delta J, J^{f_1}\delta f_1,..., J^{f_k}\delta f_k).\]

On the other hand, let $(X_i,\omega_i)$ be K\"{a}hler manifolds diffeomorphic to each other. 
\begin{defi}
Let $X$ be a compact smooth manifold. Then we define $\cJ(X)$ is the space of all almost complex structure, and $\cJ_{int}(X)$ be  the space of all integrable almost complex structure. Moreover, suppose $(X,\omega)$ be a K\"{a}hler manifold. Then we denote
\[\cJ(X,\omega):=\{J\in \cJ_{int}(X)| \omega(J\bullet, J\bullet)=\omega(\bullet,\bullet), \omega(J\bullet, \bullet)>0\}.\] 
\end{defi} 
There is a natural almost complex structure in $\cJ(X_0,\omega_0)\times \prod_{i=1}^k\cJ(X_i ))$, where $\widetilde{J} \in End(TX)\times End(TX_1)\times \cdots \times End(TX_k)$ which is defined by  \[\hat{J}|_{(J_0,\cdots , J_k)}(A_0,\cdots, A_k):= (J_0A_0,\cdots , J_kA_k).\]
By \cite{DaSm02}, $\widetilde{J}$ is indeed integrable, and it is a K\"{a}hler manifold.
Moreover, the map \[F:\cJ(X_0,\omega_0)\times \prod_{i=1}^k\map(X_0,X_i))\rightarrow \cJ(X_0,\omega_0)\times \prod_{i=1}^k\cJ(X_i)\] defined by 
 \[F(J, f_1,...,f_k)=(J, J^{f_1},...,J^{f_k})=(J, Df_1 J Df_1^{-1},..., Df_k J Df_k^{-1})\] is a smooth map satisfying
\[\widetilde{J} (DF(A, \vec{v})|_{(J,f_1,...,f_k)})= (JA, J^{f_1}v_1,..., J^{f_k}v_k)=(DF) (\hat{J}(A, \vec{v})|_{(J,f_1,...,f_k)}). \] Hence  $\cJ(X_0,\omega_0)\times \prod_{i=1}^k\map(X_0,X_i))$ can be considered as a $\widetilde{J}$ closed submanifold of $\displaystyle \cJ(X_0,\omega_0)\times \prod_{i=1}^k\map(X_0,X_i))$, and \[\cY_{\vec{p}}=F^{-1}(\cJ(X_0,\omega_0)\times \prod_{i=1}^k\cJ(X_i,\omega_i)).\]

Therefore, $\hat{J}$ is integrable as $\hat{J}=F^*\widetilde{J}$,
% (An alternate proof is at the appendix, see Lemma \ref{J is integrable})
 and hence $F$ is biholomorphic. By theorem 4 of \cite{DaSm02}, $(\cJ(X_0,\omega_0)\times \prod_{i=1}^k\cJ(X_i,\omega_i))$ is a complex manifold, hence we have:

\begin{lemm}\label{lem: Y is complex}
$\cY_{\vec{p}}$ is a complex manifold with integrable almost complex structure $\hat{J}$.
\end{lemm}
%\begin{proof}
%Consider the map
%\[F:\cJ(X_0,\omega_0)\times \cZ_{\vec{p}}\rightarrow \cJ(X_0,\omega_0)\times\prod_{i=1}^k\cJ(X_i)\] defined by 
%\[F(J, f_1,...,f_k)=(J, J^{f_1},...,J^{f_k})=(J, Df_1 J Df_1^{-1},..., Df_k J Df_k^{-1}).\] 
%Notice that \[DF(A, v_1,...,v_k)=A+\sum_{i=1}^k(Dv_i J Df_i^{-1}-Df_iJ Df_i^{-1}Dv_i Df_i^{-1})=A+\sum_{i=1}^k([Dv_iDf_i^{-1}, J^{f_k}]).\] Hence it is differentiable. 
%
%Also, 
%\[\widetilde{J} DF(A, \vec{v})|_{J,f}= (JA, J^{f_1}v_1,..., J^{f_k}v_k)=DF \hat{J}(A, \vec{v})|_{J,f}, \] so $F$ is holomorphic. Therefore, \[\cY_{\vec{p}}=F^{-1}(\cJ(X_0,\omega_0)\times \prod_{i=1}^k\cJ(X_i,\omega_i))\] is a complex submanifold of $\cJ(X_0,\omega_0)\times \cZ_{\vec{p}}$.
%\end{proof}
As a consequence, we have the following result:
\begin{theo}\label{thm: Kahler moment map of ccscK general}
$(\cY_{\vec{p}}, \Omega_{J,\vec{p}}:=\Omega_J+\Omega_{\vec{p}}, \hat{J})$ is a K{\"a}hler manifold which is closed under the action $\displaystyle\prod_{i=0}^k\ham(X_i,\omega_i)$, in which 
\begin{align*}
(\sigma_0,...,\sigma_i)\cdot(J,f_1,...,f_k)=&(\sigma_0^{-1}\cdot J, \sigma_1\circ f_1 \circ \sigma_0^{-1},\cdots, \sigma_k\circ f_k\circ \sigma_0^{-1})\\
=&(D\sigma_0JD\sigma_0^{-1}, \sigma_1\circ f_1 \circ \sigma_0^{-1},\cdots, \sigma_k\circ f_k\circ \sigma_0^{-1}).
\end{align*} 
Therefore, the moment map defined in Lemma \ref{lem: moment map of ccscK 1} can be restricted in $\cY_{\vec{p}}$. \\

We denote this moment map as ${\mu}_{\cJ,\vec{p}}$. 
\end{theo}
\begin{proof}
Let $(A, \varphi_1,...,\varphi_k), (B,\psi_1,...,\psi_k)\in T_{(J,f_1,...,f_k)}\cY_{\vec{p}}$. 	Then 
\[\Omega_{J,\vec{p}}((A, \varphi_1,...,\varphi_k), (B,\psi_1,...,\psi_k)):=\langle A,B\rangle_{\omega_0}+\sum_{i=1}^k\dfrac{(n-p_i)}{(n-p_i)!p_i!}\int_{X_i}\omega_i(\varphi_i,\psi_i) \omega_0^{n-p_i}\wedge f_i^*\omega_i^{p_i}.\] Then 
\begin{align*}
&\Omega_{J,\vec{p}}(\hat{J}(A, \varphi_1,...,\varphi_k), \hat{J}(B,\psi_1,...,\psi_k))\\=&\langle JA,JB\rangle_{\omega_0}+\sum_{i=1}^k\dfrac{(n-p_i)}{(n-p_i)!p_i!}\int_{X_i}\omega_i(J^{f_i}\varphi_i,J^{f_i}\psi_i) \omega_0^{n-p_i}\wedge f_i^*\omega_i^{p_i}\\
=&\langle A,B\rangle_{\omega_0}+\sum_{i=1}^k\dfrac{(n-p_i)}{(n-p_i)!p_i!}\int_{X_i}\omega_i(\varphi_i,\psi_i) \omega_0^{n-p_i}\wedge f_i^*\omega_i^{p_i}=\Omega_{J,\vec{p}}((A, \varphi_1,...,\varphi_k), (B,\psi_1,...,\psi_k))
\end{align*}
as $J^{f_i}\in \cJ(X_i,\omega_i)$. Hence $\Omega_{J,\vec{p}}$ is $J$ invariant, which implies it is a K{\"a}hler form.\\

For the action part, first,
 \begin{align*}&(\sigma_0,...,\sigma_i)\cdot (J,J^{f_1},\cdots, J^{f_k})\\
 =&(D\sigma_0JD\sigma_0^{-1}, D\sigma_1Df_1 D\sigma_0^{-1}D\sigma_0 JD\sigma_0^{-1}D\sigma_0Df_1^{-1}D\sigma_1^{-1},\cdots,D\sigma_kDf_k D\sigma_0^{-1}D\sigma_0JD\sigma_0^{-1}D\sigma_0Df_k^{-1}D\sigma_k^{-1})\\
 =&(D\sigma_0JD\sigma_0^{-1}, D\sigma_1Df_1 JDf_1^{-1}D\sigma_1^{-1},\cdots,D\sigma_kDf_k JDf_k^{-1}D\sigma_k^{-1}).\end{align*}
As $J^{f_i}\in\cJ(X_i,\omega_i)$, $\omega_i(J^{f_i}\bullet, J^{f_i}\bullet)=\omega_i(\bullet,\bullet)$, 
\begin{align*}
\omega_i((D\sigma_i J^{f_i}D\sigma_i^{-1})\bullet,(D\sigma_i J^{f_i}D\sigma_i^{-1})\bullet)
=&\sigma_i^*\omega_i(J^{f_i} D\sigma_i^{-1}\bullet, J^{f_i}D\sigma_i^{-1}\bullet)\\
=&\omega_i(J^{f_i} D\sigma_i^{-1}\bullet, J^{f_i}D\sigma_i^{-1}\bullet)\\
=&\omega_i(D\sigma_i^{-1}\bullet, D\sigma_i^{-1}\bullet)\\
=&\sigma_i^*\omega_i(D\sigma_i^{-1}\bullet, D\sigma_i^{-1}\bullet)\\
=&\omega_i(\bullet,\bullet).
\end{align*}
Hence $(\sigma_0,...,\sigma_i)\cdot (J,J^{f_1},\cdots, J^{f_k})\in \cY_{\vec{p}}$.\\

%Finally, we know that \[\langle \sigma_0\cdot A, \sigma_0\cdot B\rangle_{\omega_0}=\langle D\sigma_0 A D\sigma_0^{-1} ,D\sigma_0  B D\sigma_0 ^{-1}\rangle_{\omega_0}=\langle A, B\rangle_{\sigma_0^*\omega_0}=\langle A, B\rangle_{\omega_0}\] as $\sigma_0\in \ham(X_0,\omega_0)$.  Notice that $(\sigma_0,\sigma_i)\cdot \varphi_i|_x=D\sigma_i \varphi_i|_{\sigma_0^{-1}(x)}$. Hence
%\begin{align*}
%&\int_{X_i}\omega_i(D\sigma_i\varphi_i,D\sigma_i\psi_i)|_{\sigma_i\circ f_i\circ\sigma_0^{-1}(x)} \omega_0^{n-p_i}\wedge (\sigma_0^{-1})^*f_i^*\sigma_i^*\omega_i^{p_i}\\
%=&\int_{X_i}\omega_i(D\sigma_i\varphi_i,D\sigma_i\psi_i)|_{\sigma_i\circ f_i(x)} \sigma_0^*\omega_0^{n-p_i}\wedge f_i^*\sigma_i^*\omega_i^{p_i}\\
%=&\int_{X_i}\sigma_i^*\omega_i(\varphi_i,\psi_i)|_{ f_i(x)} \sigma_0^*\omega_0^{n-p_i}\wedge f_i^*\sigma_i^*\omega_i^{p_i}\\
%=&\int_{X_i}\omega_i(\varphi_i,\psi_i)|_{ f_i(x)} \omega_0^{n-p_i}\wedge f_i^*\omega_i^{p_i}.
%\end{align*}
%As a result, this is a symplectomorphsim. And the moment map defined in lemma \ref{lem: moment map of twisted equation} implies that for any $\displaystyle \xi:=(\xi_0, \cdots, \xi_k)\in Lie(\prod_{i=0}^k\ham(X_i,\omega_i))$, 
%\[(d\mu_p(A, v_1,...,v_k))(\xi)=\Omega_{\cJ,\vec{p}}(X_{\xi},(A, v_1,...,v_k)).\]  
\end{proof}

Recall that $\ham_{J}^{\CC}(X,\omega_X)$ is given as
\[\ham_{J}^{\CC}(X,\omega_X)=\{\sigma\in \map(X,X)|\varphi^*\omega_X=\omega_X+\sqrt{-1}\dd_{J}\dbar_{J} h_{\sigma}\}\] for some K{\"a}hler potential $h$. Notice that the $
\displaystyle\prod_{i=0}^k\ham(X_i,\omega_i)$ action  is closed in $\cY_{\vec{p}}$ and $\cY_{\vec{p}}$  is a complex manifold implies that the orbit space is given by  \[\cO_{J,\vec{f}}:=\left(\prod_{i=0}^k\ham_{J_i}^{\CC}(X_i,\omega_i)\right)\cdot (J,f_1,...,f_k)\] is in $\cY_{\vec{p}}$. Moreover,
\begin{align*}
&F((\sigma_0,\cdots,\sigma_k)\cdot(J,f_1,\cdots,f_k))\\
=& F(D\sigma_0 JD\sigma_0^{-1}, \sigma_1\circ f_1\circ \sigma_0^{-1},\cdots,\sigma_k\circ f_k\circ \sigma_k^{-1} )
\\ =& (D\sigma_0 JD\sigma_0^{-1},D(\sigma_1\circ f_1\circ \sigma_0^{-1})(D\sigma_0 JD\sigma_0^{-1})D(\sigma_1\circ f_1\circ \sigma_0^{-1})^{-1},\cdots, D(\sigma_k\circ f_k\circ \sigma_0^{-1})(D\sigma_0 JD\sigma_0^{-1})D(\sigma_k\circ f_k\circ \sigma_0^{-1})^{-1})\\
=&(D\sigma_0 JD\sigma_0^{-1},D(\sigma_1\circ f_1) JD(\sigma_1\circ f_1)^{-1},\cdots, D(\sigma_k\circ f_k) JD(\sigma_k\circ f_k)^{-1})\\
 =&(\varphi_0,\cdots,\varphi_k)\cdot F(J,f_1,\cdots,f_k).
\end{align*} Therefore, we have
\[\cO_{J,\vec{f}}=F^{-1}(\prod_{i=0}^k\left(\ham_{J_i}^{\CC}(X,\omega_i)\right),\]  hence $\cO_{J,\vec{f}}$ is also a K\"{a}hler submanifold of $\cY_{\vec{p}}$.  

\begin{rema}
Notice that although $(\varphi_0,\cdots,\varphi_k)\cdot (J,f_1,...,f_k)$ is well defined, $\ham_{J}^{\CC}(X,\omega_X)$ is not a group. As a remark, we can consider the orbit space as a subset of the action coming from $\displaystyle\prod_{i=0}^k\diff(X_i)$.
\end{rema}

\begin{theo}\label{thm: ccscK equation gerenal}
Consider the moment map $\mu_{\cJ,\vec{p}}:\cO_{J,\vec{f}}\rightarrow Lie(H_0\times...\times H_k)^*$ defined by Theorem \ref{thm: Kahler moment map of ccscK general} restricted on $\cO_{J,\vec{f}}$. Then $\mu_{\vec{p}}=0$ iff \[\left\{\begin{matrix}\displaystyle
\sum_{i=1}^k\left(\dfrac{f_i^*\omega_{i,\varphi_i}^{p_i+1}}{(p_i+1)!}\wedge \dfrac{\omega_{0,\varphi_0}^{n-p_i-1}}{(n-p_i-1)!}\right)-\Ric(\omega_{0,\varphi_0}, J_0)\wedge\dfrac{\omega_{0,\varphi_0}^{n-1}}{(n-1)!}-c_0\dfrac{\omega_{0,\varphi_0}^n}{n!}&=&0\\
\dfrac{\omega_{0,\varphi_0}^{n-p_1}}{(n-p_1)!}\wedge \dfrac{f_1^*\omega_{1,\varphi_1}^{p_1}}{p_1!}-c_1\dfrac{f_1^*\omega_{1,\varphi_1}^n}{n!}&=&0\\
\vdots\\
\dfrac{\omega_{0,\varphi_0}^{n-p_k}}{(n-p_k)!}\wedge \dfrac{f_k^*\omega_{k,\varphi_k}^{p_k}}{p_k!}-c_k\dfrac{f_k^*\omega_{k,\varphi_k}^n}{n!}
&=&0
\end{matrix}\right..\] In particular, if $X_0=\cdots=X_k$, $f_1=f_2=\cdots=f_k=id$, $\vec{p}=(0,...,0)$, then this is the ccscK equation with the classes fixed.
\end{theo} 
\begin{proof}
	
\begin{align*}
&\left\{\begin{matrix}\displaystyle 
\sum_{i=0}^k\dfrac{(\varphi_0^{-1})^*f_i^*\varphi_i^*\omega_i^{p_i+1}\wedge \omega_0^{n-p_i-1}}{(n-p_i-1)!(p_i+1)!}-\Ric(\omega_0, J_0^{\varphi_0^{-1}})\wedge\dfrac{\omega_0^{n-1}}{(n-1)!}-c_0\dfrac{\omega_0^n}{n!}&=&0\\
c_1\dfrac{\omega_1^n}{n!}-\dfrac{{\varphi_1}_*(f_1)_*(\varphi_0^{-1})_*\omega_0^{n-p_1}\wedge \omega_1^{p_1}}{(n-p_1)!p_1!}&=&0\\
\vdots\\
c_k\dfrac{\omega_k^n}{n!}-\dfrac{{\varphi_k}_*(f_k)_*(\varphi_0^{-1})_*\omega_0^{n-p_k}\wedge \omega_k^{p_k}}{(n-p_k)!p_k!}
&=&0
\end{matrix}\right.\\
&\left\{\begin{matrix}\displaystyle 
\sum_{i=0}^k\dfrac{f_i^*\omega_{i,\varphi_i}^{p_i+1}\wedge \omega_{0,\varphi_0}^{n-p_i-1}}{(p_i+1)!(n-p_i-1)!}-\varphi_0^*\Ric(\omega_0, J_0^{\varphi_0^{-1}})\wedge\dfrac{\varphi_0^*\omega_0^{n-1}}{(n-1)!}-c_0\dfrac{\omega_{0,\varphi_0}^n}{n!}&=&0\\
\dfrac{\varphi_0^*\omega_0^{n-p_1}\wedge f_1^*\varphi_1^*\omega_1^{p_1}}{(n-p_1)!p_1!}-c_1\dfrac{f_1^*\varphi_1^*\omega_1^n}{n!}&=&0\\
\vdots\\
\dfrac{\varphi_0^*\omega_0^{n-p_k}\wedge f_k^*\varphi_k^*\omega_k^{p_k}}{(n-p_k)!p_k!}-c_k\dfrac{f_k^*\varphi_k^*\omega_k^n}{n!}
&=&0
\end{matrix}\right.\\
&\left\{\begin{matrix}\displaystyle 
\sum_{i=0}^k\dfrac{f_i^*\omega_{i,\varphi_i}^{p_i+1}\wedge \omega_{0,\varphi_0}^{n-p_i-1}}{(p_i+1)!(n-p_i-1)!}-\Ric(\omega_{0,\varphi_0}, J_0)\wedge\dfrac{\omega_{0,\varphi_0}^{n-1}}{(n-1)!}-c_0\omega_{0,\varphi_0}^n&=&0\\
\dfrac{\omega_{0,\varphi_0}^{n-p_1}\wedge f_1^*\omega_{1,\varphi_1}^{p_1}}{(n-p_1)!p_1!}-c_1\dfrac{f_1^*\omega_{1,\varphi_1}^n}{n!}&=&0\\
\vdots\\
\dfrac{\omega_{0,\varphi_0}^{n-p_k}\wedge f_k^*\omega_{k,\varphi_k}^{p_k}}{(n-p_k)!p_k!}-c_k\dfrac{f_k^*\omega_{k,\varphi_k}^n}{n!}
&=&0.
\end{matrix}\right.
\end{align*}
Finally, if $f_i=id$, and $p_i=0$, then $f_i^*\omega_{i,\varphi_i}=\omega_{i,\varphi_i}$,  the equations become 
\[\left\{\begin{matrix}\displaystyle
\sum_{i=0}^k(\omega_{i,\varphi_i} -\Ric(\omega_{0,\varphi_0}), J_0)\wedge\omega_{0,\varphi_0}^{n-1}-c_0\omega_{0,\varphi_0}^n&=&0\\
\omega_{0,\varphi_0}^{n}-c_1\omega_{1,\varphi_1}^n&=&0\\
\vdots\\
\omega_{0,\varphi_0}^{n}-c_k\omega_{k,\varphi_k}^n
&=&0,
\end{matrix}\right.\]  which is the ccscK equation. 
\end{proof}
As a  remark, $c_i$ are constants along the whole orbit. Also, we can replace $\omega_i$ by $a_i\omega_i$, so the equation becomes 
\[\left\{\begin{matrix}\displaystyle
\sum_{i=1}^ka_i\left(\dfrac{\omega_{i,\varphi_i}^{p_i+1}}{(p_i+1)!}\wedge \dfrac{\omega_{0,\varphi_0}^{n-p_i-1}}{(n-p_i-1)!}\right)-\Ric(\omega_{0,\varphi_0}, J_0)\wedge\dfrac{\omega_{0,\varphi_0}^{n-1}}{(n-1)!}-b_0\dfrac{\omega_{0,\varphi_0}^n}{n!}&=&0\\
\dfrac{\omega_{0,\varphi_0}^{n-p_1}}{(n-p_1)!}\wedge \dfrac{\omega_{1,\varphi_1}^{p_1}}{p_1!}-b_1\dfrac{\omega_{1,\varphi_1}^n}{n!}&=&0\\
\vdots\\
\dfrac{\omega_{0,\varphi_0}^{n-p_k}}{(n-p_k)!}\wedge \dfrac{\omega_{k,\varphi_k}^{p_k}}{p_k!}-b_k\dfrac{\omega_{k,\varphi_k}^n}{n!}
&=&0,
\end{matrix}\right.\] where $b_i$ are the normalizing constants.
\subsection{An alternate setup for a special case of the coupled K{\"a}hler Yang-Mills equation}
We first construct the moment map equation described in \cite{AGG13} for $U(1)$ case. Notice that to solve the equation, we first need 
\[\mu(J,f)=\alpha_0\mu_{\cJ}(J)+\alpha_1\mu_1^*(f)+\alpha_2\mu_0^*(f)\] under a suitable subspace.
We define   
\[\Omega_{\cJ,01;\alpha_0,\alpha_1,\alpha_2}=\alpha_0\Omega_{\cJ}-\alpha_1\Omega_0^*+\alpha_2\Omega_1^*,\]
that is, for $g=f^{-1}: Y\rightarrow X $
\[(\alpha_2\Omega_1^*-\alpha_1\Omega_0^*)(\delta g_1,\delta g_2)=\int_Y \omega_X(\delta g_1,\delta g_2) (\alpha_2 \omega_Y^{[2]}\wedge g^*\omega_X^{[n-2]}-(\alpha_1)\omega_Y\wedge g^*\omega_X^{[n-1]}).\] Notice that we take dual moment map as we need 
\[\int_Y (\alpha_1\omega_Y\wedge \omega_X^{[n-1]}-\alpha_2\omega_X^{[n]})=0,\] so we cannot choose this as the K\"{a}hler form. 

Also, we take \[\cYM_{01}\subset \{(J,f,g)\in \cJ(X,\omega_X)\times \map(Y,X;n-1)^+\times \map(Y,X;n-2)^+| f=g^{-1}, J^f\in \cJ(X,\omega_X) \}\] such that 
\[\Omega_{\cJ,01;\alpha_0,\alpha_1,\alpha_2}>0 \}.\] Then we have the following proposition.
\begin{prop}\label{prop: KYM for line bundle}
$\cYM_{01}$ is K{\"a}hler and closed under the action. Moreover, if
\[\alpha_0\Omega_{\cJ}-\alpha_1\Omega_0^*+\alpha_2\Omega_1^*>0,\] then the map $\mu_{\cJ,01}:\cYM_{01}\rightarrow Lie(\ham(X,\omega_X)\times\ham(Y,\omega_Y))^*$ 
\[\mu(J,f)=\left(\begin{matrix}
\alpha_0\dfrac{\Ric(X,\omega_X, J)\wedge\omega_X^{n-1}}{(n-1)!}-\alpha_1\dfrac{\omega_X^{n-1}}{(n-1)!}\wedge f^*\omega_Y+\alpha_2 \dfrac{\omega_X^{2}\wedge f^*\omega_Y^{n-2}}{(n-2)!2!}+z\omega_X^n \\
+\alpha_2\dfrac{f_*\omega_X^{n-1}}{(n-1)!}\wedge \omega_Y-\alpha_1\dfrac{f_*\omega_X^n}{n!}
\end{matrix}\right)\] is a moment map, 
where $z=\dfrac{\bar{S}}{2}-c_{10}\alpha_1-c_{11}\alpha_2$,  and here we choose $\alpha_1$ such that  $\alpha_1-\alpha_2=0$. As a corollary, $\mu_{\cJ,01}=0$ iff 
\[\left\{\begin{matrix}
\alpha_0\dfrac{\Ric(X,\omega_X, J)\wedge\omega_X^{n-1}}{(n-1)!}+\alpha_2 \dfrac{\omega_X^{2}\wedge f^*\omega_Y^{n-2}}{(n-2)!2!}=c\dfrac{\omega_x^n}{n!} \\
\dfrac{f_*\omega_X^{n-1}}{(n-1)!}\wedge \omega_Y=d\dfrac{f_*\omega_X^n}{n!}
\end{matrix}\right.,\] where $d=\dfrac{\alpha_1}{\alpha_2}$, $c=\alpha_1d+z$. 
\end{prop}
\begin{proof}
	Notice that $\cYM_{01}$ is a submanifold of $\cY_0$, so the complex structure is defined directly by $\cY_0$. Then $\Omega_{\cJ}$ is $J$ invariant and $\Omega_0$ is $J^f$ invariant. Also, if we  define $inv:\map(X,Y)\rightarrow \map(Y,X)$, and we define $J'$ on $\map(Y,X)$ such that  \[inv_*( J^f\delta f)=J' inv_*\delta f,\] then 
	\[J'Df^{-1}= Df^{-1}J^f=JDf^{-1}.\] That means the map $inv$ is a biholomorphism, and $\Omega^*$ is also K\"{a}hler if $\Omega$ is. As $\omega_X$ is $J$-invariant, $\Omega_1$ is also $J$-invariant, hence 
	\[\Omega_{\cJ,01;\alpha_0,\alpha_1,\alpha_2}=\alpha_0\Omega_{\cJ}-\alpha_1\Omega_0^*+\alpha_2\Omega_1^*\]  is $J$-invariant, which implies $\cYM_{01}$ is K{\"a}hler.
	
	Moreover, the moment map is given by
 \begin{align*}
&\mu_{\cJ,{01}}(J,f)\\
=&\left(\begin{matrix}
\alpha_0\dfrac{\Ric(X,\omega_X, J)\wedge\omega_X^{n-1}}{(n-1)!}-\alpha_1\dfrac{\omega_X^{n-1}}{(n-1)!}\wedge f^*\omega_Y+\alpha_2 \dfrac{\omega_X^{2}\wedge f^{-1}_*\omega_Y^{n-2}}{(n-2)!2!}-\left(\dfrac{\bar{S}}{2}-c_{10}\alpha_1-c_{11}\alpha_2\right)\omega_X^n \\
\alpha_2\dfrac{{f^{-1}}^*\omega_X^{n-1}}{(n-1)!}\wedge \omega_Y-\alpha_1\dfrac{f_*\omega_X^n}{n!}-\left(c_{20}\alpha_1-c_{21}\alpha_2\right)\dfrac{\omega_Y^n}{n!}
\end{matrix}\right)\\
=& \left(\begin{matrix}
\alpha_0\dfrac{\Ric(X,\omega_X, J)\wedge\omega_X^{n-1}}{(n-1)!}-d\dfrac{\omega_X^{n-1}}{(n-1)!}\wedge f^*\omega_Y+\alpha_2 \dfrac{\omega_X^{2}\wedge f^*\omega_Y^{n-2}}{(n-2)!2!}+z\omega_X^n \\
\alpha_2\dfrac{f_*\omega_X^{n-1}}{(n-1)!}\wedge \omega_Y-d\dfrac{f_*\omega_X^n}{n!}
\end{matrix}\right),
\end{align*} 
where $z=\dfrac{\bar{S}}{2}-c_{10}\alpha_1-c_{11}\alpha_2$,  and we choose $d=\alpha_1$ such that  $c_{20}\alpha_1-c_{21}\alpha_2=0$.

The last part is obvious.
\end{proof}
Hence we get the same  moment map equation for the K{\"a}hler Yang-Mill's equation with $G=U(1)$ case (see \cite{AGG13} for general). As a remark, we can easily generalize it into $U(1)^n$ case. We can generalized this moment map by considering the following equation: 
\[\Omega_{\cJ,pq}:=\alpha_0\Omega_{\cJ}-\dfrac{\alpha_1(n-p)}{n}\Omega_p^*+\dfrac{\alpha_2(q+1)}{n}\Omega_q^*.\] 
Define  \[\cYM_{pq}\subset \{(J,f,g)\in \cJ(X,\omega_0)\times\map(Y,X;n-p)^+\times \map(Y,X;n-q-1)^+|g=f^{-1}, J^f\in\cJ(Y,\omega_Y) \}\] such that $\Omega_{\cJ,pq}>0$. That is, $f\in \map(X,Y;p)^+\cap\map(X,Y;q)^+$. Then $\cYM_{pq}$ is closed under the action of $\ham(X,\omega_X)\times\ham(Y,\omega_Y)$, and the map \[\mu_{\cJ,p,q,\alpha_0,\alpha_1,\alpha_2}(J,f):
=\alpha_0\mu_{\cJ}(J)-\dfrac{\alpha_1(n-p)}{n}\mu_p(f)+\dfrac{\alpha_2(q+1)}{n}\mu_q(f)\] is the moment map for 

\begin{align*}
\mu_{\cJ,pq}(J,f)=&\left(\begin{matrix}
\alpha_0\dfrac{\Ric(X,\omega_X, J)\wedge\omega_X^{n-1}}{(n-1)!}+\alpha_1\dfrac{\omega_X^{n-p-1}}{(n-p-1)!}\wedge \dfrac{f^*\omega_Y^{p+1}}{(p+1)!}-\alpha_2 \dfrac{\omega_X^{n-q-1}\wedge f^{-1}_*\omega_Y^{q+1}}{(n-q-1)!(q+1)!}+z\omega_X^n \\
-\alpha_2\dfrac{{f^{-1}}^*\omega_X^{n-q}}{(n-q)!}\wedge \dfrac{\omega_Y^q}{q!}+\alpha_1\dfrac{f_*\omega_X^{n-p}}{(n-p)!}\wedge\dfrac{\omega_Y^{p}}{p!}-\left(c_{20}\alpha_1-c_{21}\alpha_2\right)\dfrac{\omega_Y^n}{n!}
\end{matrix}\right).
\end{align*}
Moreover, we can choose $\alpha_1$ such that $c_{20}\alpha_1-c_{21}\alpha_2=0$. \\

\subsection{Coupled DHYM types equation}
In \cite{ScSt19}, Schlitzer and Stoppa studied  coupled Deformed Hermitian Yang-Mills equation using Extended Gauged group theory. We now using the theory in this note to recover the coupled  DHYM equation.\\

Recall that the DHYM equation is given by the following: Let $(X,\omega,L)$ be a projective manifold, and $\alpha=\sqrt{-1}F(L)$. Then the DHYM is given by 
\[Im(e^{\sqrt{-1}\theta}(\omega+\sqrt{-1}\alpha))^n=0\] 
with $Re(e^{\sqrt{-1}\theta}(\omega+\sqrt{-1}\alpha))^n>0$.
Here $\theta$ is some constant defined by the class of $\omega$ and $\alpha$. 
Expend the expression $e^{\sqrt{-1}\theta}(\omega+\sqrt{-1}\alpha)^n$, we get 
\[\text{Imarginay part}: \cos\theta \sum_{r=0}^k (-1)^rC_{2r+1}^n\omega^{n-2r-1}\wedge \alpha^{2r+1}+\sin\theta  \sum_{r=0}^{k}(-1)^{r}C_{2r}^n\omega^{n-2r}\wedge \alpha^{2r}=0;\]
\[\text{Real part}: \cos\theta \sum_{r=0}^{k}(-1)^{r}C_{2r}^n\omega^{n-2r}\wedge \alpha^{2r}-\sin\theta \sum_{r=0}^k (-1)^rC_{2r+1}^n\omega^{n-2r-1}\wedge \alpha^{2r+1}>0.\]
Here $k$ is the value such that the $2k=n-1$ or $n$.  

Under the previous construction, consider
\[\sum_{r=0}^{k}(-1)^r\cos\theta C_{2r}^n\mu_{2r}-\sin\theta\sum_{s=0}^l(-1)^{r}C_{2r+1}\mu_{2r+1}, \]
where $k$ is chosen such that $2k\leq n-1$, $2l+1\leq n-1 $  under the domain 
\[\cY_{dHYM}\subset \bigcap_{p=0}^{n-1} \map(X,Y;p)^+\] such that 
\[\cos\theta \sum_{r=0}^{k}(-1)^{r}C_{2r}^n \omega^{n-2r}\wedge f^*\alpha^{2r}-\sin\theta \sum_{r=0}^k (-1)^rC_{2r+1}^n\omega^{n-2r-1}\wedge f^*\alpha^{2r+1}>0.\]

Suppose this space is non empty, 
the equation is given by 
\[\begin{matrix}
\displaystyle \cos\theta \sum_{r=0}^{k}(-1)^{r}C_{2r}^n \omega^{n-2r-1}\wedge f^*\alpha^{2r+1}-\sin\theta \sum_{r=0}^{k-1} (-1)^rC_{2r+1}^n\omega^{n-2r-2}\wedge f^*\alpha^{2r+2} &=& c_1\omega^n\\
\displaystyle  \cos\theta \sum_{r=0}^{k}(-1)^{r}C_{2r}^n f_*\omega^{n-2r}\wedge \alpha^{2r}-\sin\theta \sum_{r=0}^k (-1)^rC_{2r+1}^n f_*\omega^{n-2r-1}\wedge \alpha^{2r+1} &=& c_2\alpha^n
\end{matrix}\]
Rewrite it, we get 
\[\begin{matrix}
\displaystyle  \cos\theta \sum_{r=0}^{k}(-1)^{r}C_{2r}^n \omega^{n-2r-1}\wedge f^*\alpha^{2r+1}+\sin\theta \sum_{r=1}^{k} (-1)^rC_{2r+1}^n\omega^{n-2r}\wedge f^*\alpha^{2r} &=&c_1\omega^n\\
\displaystyle  \cos\theta \sum_{r=0}^{k}(-1)^{r}C_{2r}^n f_*\omega^{n-2r}\wedge \alpha^{2r}-\sin\theta \sum_{r=0}^{k} (-1)^rC_{2r+1}^n f_*\omega^{n-2r-1}\wedge \alpha^{2r+1} &=& c_2\alpha^n
\end{matrix},\]
where $k$ is the value such that the $2k=n-1$ or $n$. 
Notice that the 2 form
\[\Omega(\delta f_1, \delta f_2)= \int_X \alpha(\delta f_1, \delta f_2)\left(\cos\theta \sum_{r=0}^{k}(-1)^{r}C_{2r}^n f_*\omega^{n-2r}\wedge \alpha^{2r}-\sin\theta \sum_{r=0}^{k} (-1)^rC_{2r+1}^n f_*\omega^{n-2r-1}\wedge \alpha^{2r+1}\right)\] define a positive symplectic form iff 
\[\cos\theta \sum_{r=0}^{k}(-1)^{r}C_{2r}^n\omega^{n-2r}\wedge \alpha^{2r}-\sin\theta \sum_{r=0}^k (-1)^rC_{2r+1}^n\omega^{n-2r-1}\wedge \alpha^{2r+1}>0.\]
Hence, when this $\Omega$, we get the domain of the moment map. If we also restrict the subgroup to be  $\ham(X,\omega)$,  it is the DHYM equation. So we can recover a moment map set up in \cite{CXY17}. However, we cannot recover the coupled DHYM using this moment map as we will couple the scalar curvature with the imginary part, not the real part. 

To recover the setup of the coupled deformed HYM, we consider another setup, namely, 
\[\mu_{\cJ}+\cos\theta \sum_{s=0}^l (-1)^sC_{2s+1}^n\mu_{2s+1}+\sin\theta  \sum_{r=0}^{k}(-1)^{r}C_{2r}^n\omega^{n-2r}\mu_{2r};\]
where $k,l$ is chosen such that $2k\leq n-1$, $2l+1\leq n-1 $. Denote the space as $\cY_{dHYM}'$ similar to the definition of $\cY_{dHYM}$, and we can define 
\[\cY_{cdHYM}\subset \cJ_{int}\times \cY_{dHYM}'\] to be the largest K\"{a}hler submanifold which is closed under the orbit similar to the setup of gerenal ccscK. Then the resulting moment map equation is given by 
\[\begin{matrix}
\displaystyle  Ric(\omega,J)\wedge\omega^{n-1}+ \cos\theta \sum_{r=0}^{k}(-1)^{r}C_{2r}^n \omega^{n-2r}\wedge f^*\alpha^{2r}-\sin\theta \sum_{r=0}^{k} (-1)^rC_{2r+1}^n \omega^{n-2r-1}\wedge f^*\alpha^{2r+1} &=& c_1\omega^n\\
\displaystyle  \cos\theta \sum_{r=0}^{k}(-1)^{r}C_{2r}^n f_*\omega^{n-2r-1}\wedge \alpha^{2r+1}+\sin\theta \sum_{r=1}^{k} (-1)^rC_{2r+1}^nf_*\omega^{n-2r}\wedge \alpha^{2r} &=&c_2\alpha^n
\end{matrix}\]
if $c_2$ is positive. In particular, if  we consider the orbit space \[ (\ham(X,\omega)\times \ham(X,\alpha))^{\CC}\cdot \{f_0=id\},\] then the equation can be reformulated as
\[\begin{matrix}\displaystyle  Ric(\omega_{\varphi})\wedge\omega_{\varphi}^{n-1}+ \cos\theta \sum_{r=0}^{k}(-1)^{r}C_{2r}^n \omega_{\varphi}^{n-2r}\wedge \alpha_{\psi}^{2r}-\sin\theta \sum_{r=0}^{k} (-1)^rC_{2r+1}^n \omega_{\varphi}^{n-2r-1}\wedge \alpha_{\psi}^{2r+1} &=& c_1\omega_{\varphi}^n\\
\displaystyle  \cos\theta \sum_{r=0}^{k}(-1)^{r}C_{2r}^n \omega_{\varphi}^{n-2r-1}\wedge \alpha_{\psi}^{2r+1}+\sin\theta \sum_{r=1}^{k} (-1)^rC_{2r+1}^n\omega_{\varphi}^{n-2r}\wedge \alpha_{\psi}^{2r} &=&c_2\alpha_{\psi}^n
\end{matrix}.\]
Finally, to avoid the sign problem, we may replace all $-\mu_r$ into $\mu_r^*$, then we can make sure the $\Omega$ is positive.
\section{Application}\label{sec: application}

\subsection{Obstructions on solving generalized ccscK}
For moment maps on the complexified orbit, there are some standard results (see \cite{Wang04}). For example, we can define the Futaki invariant, Calabi functional and Mabuchi functional that can provide some obstructions of the moment map equation $\mu=0$ (see \cite{Don01}, \cite{Don02}, \cite{PS2009} for cscK, \cite{DaPi19} for ccscK, and \cite{AGG13} for K{\"a}hler Yang Mill). We will consider the generalized ccscK equation 
\[\mu_{\cJ,\vec{p}}:\cO_{\cJ,id}\rightarrow Lie\left(\prod_{i=0}^k\ham(X_i,\omega_i)\right)^*.\]

For fix $f_i$, we can define a map $f_i^*:\diff(X_i, X_i)\rightarrow \diff(X_0,X_0)$ by 
\[f_i^*\varphi:= f_i^{-1}\circ\varphi \circ f_i.\] We also denote $(f_i)_*=(f_i^{-1})^*$. Then we can define 
\[G_0^{\vec{f}}:=Aut(X_0,L_0)\cap \cap_{i=1}^kf_i^*Aut(X_i,L_i),\] and 
\[G_j^{\vec{f}}=(f_j)_*G_0^{\vec{f}}.\] 

\begin{lemm}
$G_j^{\vec{f}}$ are subgroup of $Aut(X_j,L_j)$. Moreover, the embedding map 
\[\displaystyle\iota: G_0^{\vec{f}}\rightarrow \prod_{i=0}^kAut(X_i,L_i)\] defined by 
\[\iota(\varphi)=(\varphi, (f_1)_*\varphi,\cdots, (f_k)_*\varphi)\] is an homomorphism, and $G_0^{\vec{f}}$ is the stabilizer of $(J,\vec{f})$ as a subgroup of $\displaystyle\prod_{i=0}^kAut(X_i,L_i)$.
\end{lemm}  
\begin{proof}
 Let $\varphi,\psi \in f_i^*Aut(X_i,L_i)$ Then $f_i \circ\varphi\circ f_i^{-1}, f_i\circ\psi\circ f_i^{-1}\in Aut(X_i,L_i)$. Then \[(f_i \circ\varphi\circ f_i^{-1})\circ(f_i\circ\psi\circ f_i^{-1})^{-1}=f_i\circ\varphi\circ\psi^{-1}\circ f_i^{-1}\in Aut(X_i,L_i),\] hence $\varphi\circ \psi^{-1}\in f_i^*Aut(X_i,L_i)$.

For the second part, first, 
\[\pi_i(\iota(\varphi)\iota(\psi)^{-1})=(f_i\circ \varphi \circ f_i^{-1})\circ (f_i\circ \psi \circ f_i^{-1})^{-1}=f_i\circ\varphi\circ\psi^{-1}\circ f_i^{-1}=\pi_i(\iota(\varphi\circ\psi^{-1})).\] 

It is well known that we can identify $Aut(X_0, L_0)$ with $G_J$.  We can identify $f\in\mathfrak{g}_J$ with $\xi_f\in \mathfrak{aut}(X_0,L_0)$  defined by 
\[\mathfrak{g}_J:=\{f\in \mathfrak{g}_J^{\CC}| \dbar \xi_{f}=0, \iota_{\xi_f}\omega=df.\}\]  Also, for $\varphi\in G_0^{\vec{f}}$, 
\[((f_i)_*\varphi)\cdot f_i= f_i\circ \varphi\circ f_i^{-1}\circ f_i\circ \varphi^{-1}=f_i.\] 

Finally, if $\displaystyle (\varphi_0,\cdots,\varphi_k)\in \prod_{i=0}^kAut(X_i,L_i)$ such that $\varphi_i\circ f_i\circ \varphi_0^{-1}=f_i$, then 
\[\varphi_i= f_i\circ \varphi_0\circ f_i^{-1},\] which implies $\varphi_0\in G_0^{\vec{f}}$. 
\end{proof}

\begin{coro}\label{cor: Aut is reductive}
	Suppose $(X_i,L_i)$ is a projective manifold with line bundles with their respective curvatures $\omega_0,...,\omega_k$. Suppose generalized ccscK has a solution, then $\displaystyle \bigcap_{i=0}^kAut(X_i,L_i)$ is reductive.
\end{coro}
\begin{proof}
	We use the result in \cite{Wang04}, corollary 12. Suppose $\mu(J_X^{\varphi_0^{-1}},\varphi_0,\varphi_1,...,\varphi_k)=0$ has a solution. Then $G_{\varphi_0,\varphi_1,...,\varphi_k}^{\CC}$ is reductive. By assuming $(\omega_{0,h_0},\cdots,\omega_{k,h_k})$ be the solution, we have $f_i=id$ and 
	\[\displaystyle G_0^{\vec{f}}=G_0^{id}=\bigcap_{i=0}^kAut(X_i,L_i)\]   is reductive.
\end{proof}

We can also define the Calabi functional, Futaki invariant and Mabuchi functional as follow. 
\begin{defi}\label{def: Futaki invariant}
Let $\xi=(\xi_0,...,\xi_{k})$ be an $\CC^*$ action on $\displaystyle\prod_{i=0}^kAut(X_i,L_i)$, where 
\[\iota_{\xi_i}\omega_i=dh_i, \dbar_{J_i}{\xi_i(t)}=0.\] Then the Futaki  invariant for the moment map defined in Theorem \ref{thm: ccscK equation gerenal} is defined by 
\begin{align*}
F_{\cJ,\vec{p}}(\xi):=&\langle \mu_{\cJ,\vec{p}}(f), \xi\rangle\\
=&\int_{X_0}h_0\left(\sum_{i=1}^k \dfrac{f_{i,0}^*\omega_i^{p_i+1}\wedge \omega_0^{n-p_i-1}}{(p_i+1)!(n-p_i-1)!}-\Ric(\omega_0)\wedge \dfrac{\omega_0^{n-1}}{(n-1)!}-c_0\dfrac{\omega_0^n}{n!}\right)\\
&+\sum_{i=1}^k\int_{X_i}h_i\left(\dfrac{\omega_i^{p_i}\wedge {f_{i,0}}_*\omega_0^{n-p_i}}{p_i!(n-p_i)!}-c_i\dfrac{\omega_i^n}{n!}\right).
\end{align*}
 The Futaki invariant of the ccscK equation is the case  $f_{i,0}=id$ and $p_i=0$ for all $i=1,...,k$.
\end{defi}
Again, by the standard result (for example, see proposition 6 in \cite{Wang04}, theorem 3.9 in \cite{LSW22} for the independence; or see \cite{Fut83} for the KE case), we have
\begin{coro}\label{cor: solution exists implies Futaki invariant vanish}
Futaki invariant is independent of the choice of $\omega_i$ with the given class. Moreover, if the Futaki invariant is non zero for some holomorphic vector field, then this moment map equation has no solution in the given K{\"a}hler classes. 
\end{coro}

Besides, we can define the Calabi functional, which is  $||\mu||^2$ .

By \cite{Wang04}, corollary 13, we have the following: 
\begin{coro}\label{cor: extremal metric and ccscK}
We define the extremal metric corresponding to $\mu_{\cJ,\vec{p}}$ to be the critical point of $\cC_{\cJ,\vec{p}}$. Then the extremal metric solves $\mu_{\cJ,\vec{p}}=0$ (in the domain $\cO_{J, id}$) iff the Futaki invariaant are zero for all holomorphic vector field. 
\end{coro}

 Let $K$ be a Lie group, and $K^{\CC}$ be the complexify orbit. Suppose $K$ acts on a space $X$ with a hamilitonian group action,  and $\mu: X\rightarrow Lie(K)^*$, we can define a $K$ invariant one form on $(K^{\CC}/K)$, defined by the following: for any $v\in Lie(K)$, 
\[\alpha\left(\dfrac{d}{dt}e^{-\sqrt{-1}tv}\cdot g\right):=\langle \mu(g\cdot z), v\rangle.\] 
 It is well-defined and independent of the choice of $g' \in K\cdot g$ as
\[\langle\mu(kg\cdot z),\Ad_k\xi \rangle=\langle\Ad_k\mu(g
\cdot z), Ad_k\xi\rangle=\langle \mu(g\cdot z),\xi\rangle.\]
\begin{lemm}
$\alpha$ is closed. Therefore, it is an exact form, and hence,  there is a functional $\cM:K^{\CC}/K \rightarrow \RR$ defined by 
\[\cM(g):=\int_{0}^1\alpha(g_t)dt,\] where $g_t$ is any curve connecting a fix point $g_0$ and $g$.
\end{lemm}
\begin{proof}
	Assume $[\xi,\eta]=0$, then by identifying $Lie(K)$ and $Lie(K)^*$, 
	\begin{align*}
	d\alpha(\xi,\eta)=&\left\langle \left.\dfrac{d}{dt}\right|_{t=0}\mu\left(e^{-\sqrt{-1}t\xi}\cdot z\right), \eta\right\rangle-\left.\dfrac{d}{dt}\right|_{t=0}\left\langle \mu\left(e^{\sqrt{-1}t\eta}\cdot z\right), \xi\right\rangle\\
	=& \langle d\mu (JX_{\xi}),\eta\rangle- \langle d\mu (JX_{\eta}),\xi\rangle\\
	=& \omega(JX_{\xi},X_{\eta})-\omega(JX_{\eta},X_{\xi})\\
	=&0,
	\end{align*}
	where $X_{\xi}$ is the vector field $ \left.\dfrac{d}{dt}\right|_{t=0}e^{-\sqrt{-1}t\xi}\cdot z$. Therefore it is closed.
\end{proof}

As a result, given a moment map, we can define the Mabuchi functional by  
\[\cM_{\mu}(g):=\int_0^1\alpha(\dot{g_t})dt,\] where $g_0=id$. 

In our case, we can define $K^{\CC}$ as a complex manifold (the orbit space). Notice that \[K^{\CC}/K \cdot {id}\cong \prod_{i=0}^kPSH(X_i,\omega_i),\] so we can define the Calabi functional and Mabuchi functional by the following. 
\begin{defi}
	Let $(X_i,\omega_i)$ be K{\"a}hler manifold, then we denote 
	\[H_j^{i,p}(h_i,h_j):=\dfrac{n!}{(n-p)!p!}\dfrac{\omega_{i,h_i}^{n-p}\wedge\omega_{j,h_j}^p}{\omega_{j,h_j}^n},\] and the mean is defined by
\[\underline{H_j^{i,p}}(h_i,h_j):=\int_X H_j^{i,p}(h_i,h_j)\dfrac{\omega_j^{n}}{n!}.\]
	\end{defi}
As
\[Lie(\ham(X_i,\omega_i))\cong\{\varphi_i\in C^{\infty}(X)| d\varphi=\iota_{X_{\varphi}}\omega_i\}/\RR,\] 
and 
\[\int_{X_j}\varphi_j \dfrac{\omega_j^{n-p}\wedge f^*\omega_i^{p}}{(n-p)!(p)!} =\int_{\sigma_j(X_j)}\varphi_j(\sigma_j(x))\left.\dfrac{\omega_{j,h_j}^{n-p}\wedge \omega_{i,\sigma_i}^{p}}{(n-p)!!}\right|_{\sigma_j(x)}=\int_{X_j}\varphi_i H_0^{i,p_i+1}(h_0,h_i)\dfrac{\omega_{0,h_0}^n}{n!},\] the explicit formula of Calabi and Mabuchi functional is given by the following:
\begin{defi}\label{def: Calabi functional and Mabuchi functional}
	The Calabi functional $\displaystyle\cC_{\cJ,\vec{p}}:\prod_{i=0}^k\psh(X_i,\omega_i)\rightarrow \RR$ is defined by the formula
	\begin{align*}
	\cC_{\cJ,\vec{p}}(\vec{h})=&|| \mu_{\cJ,\vec{p}}(\vec{\sigma_h})||^2\\\
	=&\int_{X_0}\left|\sum_{i=1}^kH_0^{i,p_i+1}(h_0,h_i)-S_{h_0}-\sum_{i=1}^k\underline{H_0^{i,p_i+1}(h_0,h_i)}+\underline{S_{h_0}}\right|^2\dfrac{\omega_{0,h_0}^n}{n!}\\
	&+\sum_{j=1}^k\int_{X_j} |H_{j}^{0,n-p_i}(h_0,h_i)-\underline{H_{j}^{0,n-p_i}}(h_0,h_i)|^2\dfrac{\omega_{j,h_j}^n}{n!}.
	\end{align*}
	
The Mabuchi functional corresponding to $\mu_{\cJ,\vec{p}}$ is given by
\[\cM_{\cJ,\vec{p}}: \prod_{i=0}^k\psh(X_i,\omega_i)\rightarrow \RR\] such that the variational formula is 
\begin{align*}
d\cM_{\cJ,\vec{p}}|_{h_0,...,h_k}(\vec{\varphi})
:=&\langle\mu_{\cJ,\vec{p}}((\sigma_0,\cdots,\sigma_k)\cdot (J,\vec{f})), \vec{\xi_{\varphi}}\rangle\\
=&\int_{X_0}\varphi_0\left(\sum_{i=1}^kH_0^{i,p_i+1}(h_0,h_i)-S_{h_0}+\underline{S}-\sum_{i=1}^k\underline{H_0^{i,p_i+1}(h_0,h_i)}\right)\dfrac{\omega_{0,h_0}^n}{n!}\\
&+\sum_{j=1}^k\int_{X_j} \varphi_j\left(H_{j}^{0,n-p_i}(h_0,h_i)-\underline{H_{j}^{0,n-p_i}}(h_0,h_i)\right)\dfrac{\omega_{j,h_j}^n}{n!},
\end{align*}
where $\varphi\in C^{\infty}(X,\RR)$.

\end{defi}
Following the standard result of moment map on the comlex orbit (for example, see \cite{Don02},\cite{Wang04}), as the geodesic is given by $e^{-\sqrt{-1}t\xi}\cdot g$, 
\[\cM'(t)=\langle \mu(e^{-\sqrt{-1}t\xi}\cdot z), \xi\rangle,\]
\[\cM''(t)=\omega(-J\xi,\xi)=||\xi||^2>0.\] 

We have the following corollary.
\begin{coro}\label{cor: unique solution}
$\cM_{\cJ,\vec{p}}$ is convex along smooth geodesics. Hence the solution of the generalized ccscK is the minimum of  $\cM_{\cJ,\vec{p}}$.
\end{coro}

Notice that by \cite{Dar14}, not any two K\"{a}hler potential can be connected by the smooth geodesic in general, not even the limit of a sequence of smooth geodesic. Therefore, in general, the convexity of Mabuchi functional for smooth geodesic cannot imply the critical point is unique. However, under some special case, we will still have uniqueness result directly.

Let $(X,L_0,\cdots, L_k)$ to be a polarized toric manifold and the curvature of the toric equivariant line bundle $L_i$ is $\omega_i$ which are positive. Let $P_i$ be the moment polytopes corresponding to $L_i$. We also denote$P_i$ is defined by the equations 
\[\cap_{\alpha} \{l_i^{\alpha}(x)\geq 0\},\] where $l_i^{\alpha}(x)$ are affine functions.  Recall that (See \cite{Gua99} \cite{Don02}, \cite{Gue14}), the space of the $(S^1)^n$ invariant K\"{a}hler form with the K\"{a}hler class $[\omega_i]$ 
\[\{\varphi\in C^{\infty}(X,\RR)|\omega_i+\ddbar\varphi>0, \varphi(\theta\cdot x)=\varphi(x), \theta\in (S^1)^n\} \] is isometric to the space 
\[\cH_i:=\{u\in C^{\infty}(P_i^0)| u\text{ is convex},u_i=\sum_{\alpha} (l_i^{\alpha}(x)\log(l_i^{\alpha}(x))), u-u_i\in C^{\infty}(P_i)\},\] with the geodesic is given by $u+tv$, $v\in C^{\infty}(P_i, \RR)$. Therefore, the orbit space is isometric to the space
\[\cH_0\times \cdots \times \cH_k\] which is geodesicly convex. Therefore, as a direct consequence of \ref{cor: unique solution}, if we have two minimum point, we can connect them by a stricly convex geodesic, which lead a contradiction. Therefore,  we have

% However, in toric case, as in \cite{Don02}, under the Legendre transform, the space of $(S^1)^n$ invariant K\"{a}hler potential is isometric to the space of smooth function plus the function $\varphi_0(x)=\sum_i l_i(x)\log l_i(x)$ with the geodesic to be the striaght line. Therefore, any two smooth $(S^1)^n$ invariant K\"{a}hler metric in the same K\"{a}hler class is connected by a smooth geodesic. As a result, we have:  
%
\begin{coro}\label{unique in toric} Let $(X,L_0,\cdots, L_k)$ to be a polarized toric manifold and the curvature of the toric equivariant line bundle $L_i$ is $\omega_i$. Then the $(S^1)^n$ invariant solution of the equation  \[\left\{\begin{matrix}\displaystyle
\sum_{i=1}^k\left(\dfrac{f_i^*\omega_{i,\varphi_i}^{p_i+1}}{(p_i+1)!}\wedge \dfrac{\omega_{0,\varphi_0}^{n-p_i-1}}{(n-p_i-1)!}\right)-\Ric(\omega_{0,\varphi_0}, J_0)\wedge\dfrac{\omega_{0,\varphi_0}^{n-1}}{(n-1)!}-c_0\dfrac{\omega_{0,\varphi_0}^n}{n!}&=&0\\
\dfrac{\omega_{0,\varphi_0}^{n-p_1}}{(n-p_1)!}\wedge \dfrac{f_1^*\omega_{1,\varphi_1}^{p_1}}{p_1!}-c_1\dfrac{f_1^*\omega_{1,\varphi_1}^n}{n!}&=&0\\
\vdots\\
\dfrac{\omega_{0,\varphi_0}^{n-p_k}}{(n-p_k)!}\wedge \dfrac{f_k^*\omega_{k,\varphi_k}^{p_k}}{p_k!}-c_k\dfrac{f_k^*\omega_{k,\varphi_k}^n}{n!}
&=&0
\end{matrix}\right.\] is unique (if exists).  
\end{coro}

\section{ K{\"a}hler construction for coupled equation p}\label{sec: Kahler coupled equation p}

In this section, we will try to find a suitable space for the coupled equation p which is a K{\"a}hler manifold. However, the case is much more subtle then the pervious case. The problem is, unlike $\cJ \times \map(X,Y;p)^+$, it is not easy to find a good complex submanifold inside $\map(X,Y;p)^+$ such that both $f^*\omega_Y$ and $\omega_X$ form of $X$ and $f_*\omega_X$ and $\omega_Y$ are Kahler forms of $Y$. Before we go on our main disscusion, notice that if we restrict the group to be either $\ham(X,\omega_X)$ or $\ham(Y,\omega_Y)$ we do have a good complex submanifold. For the first subgroup, the coupled moment map equation become \[\omega_X^{n-p-1}\wedge f^*\omega_Y^{p+1}=c_1\omega_X^n.\] The second subgroup gives 
\[f_*\omega_X^{n-p}\wedge \omega_Y^p=c_2\omega_Y^n.\] We have seen this trick when we re-construct the deformed HYM.

The method we suggest is the following: we consider $F_f:\map(X,X)\times \map(Y,Y)\rightarrow \map(X,Y)$ by 
\[F_f(\sigma,\eta)=\eta \circ f\circ \sigma^{-1},\] and consider the pull back image $F_f^{-1}(\map(X,Y;p)^+)$. Notice that $F_f^{-1}(\Omega_p)$ is not a symplectic form as it may be degenerated. Then we can find a "lagrest complex submanifold" $\cX_p^+$, and the orbit space $\ham^{\CC}(X,\omega_X)\times \ham^{\CC}(Y,\omega_Y)$  inside $F^{-1}(\map(X,Y;p)^+)$. And we will show that if $F_f(\cX_p^+)$ and $F_f(\ham^{\CC}(X,\omega_X)\times \ham^{\CC}(Y,\omega_Y))$ are complex manifold, then these are the spaces for the moment map picture for moment map $p$.

Let $J_X$ to be an integrable almost complex structure of $X$. Then for any diffeomorphism $g:X\rightarrow Y$, we can define an almost complex structure of $Y$ by 
\[J_Y:=J_X^g:=Dg J_XDg^{-1}.\] This is integrable as the complex local coordinate of $Y$ can be defined by $X$ and $g$, namely, if \[\{U_i,\varphi_i:U_i\rightarrow \Omega_i\subset \CC^n\}\] are complex local coordinate of $X$, then  $\{g(U_i),\psi_i:=\varphi_i\circ g^{-1}:g(U_i)\rightarrow \Omega_i\}$  with transition map 
\[\psi_j\circ \psi_i^{-1}|_{\psi_i(g^{-1}(U_i\cap U_j))}=\varphi_j\circ \varphi_i^{-1}|_{\varphi_i(U_i\cap U_j)}\] defines the complex local coordinate of $Y$.

However, let 
$(X,\omega_X,J_X), (Y,\omega_Y,J_Y)$ be two K{\"a}hler manifold. Notice that $J_X$ is compatible with $\omega_X$ doesn't implies $J_X^g$ is compatible with $\omega_Y$.

\begin{defi}
	Let $(X,\omega_X,J_X)$, $(Y,\omega_Y,J_Y)$ be compact K{\"a}hler manifolds. 
	Define \[\cJ(X,\omega_X):=\{J\in\cJ_{int}(X)| \omega_X(J\cdot, J\cdot)=\omega_X(\cdot,\cdot), \omega_X(J\cdot, \cdot)>0\}.\]
\end{defi}

Define $F_f:\map(X,X)\times\map(Y,Y)\rightarrow\map(X,Y)$ to be 
\[F_f(\varphi,\psi):=\psi\circ f\circ \varphi^{-1}.\]
We also denote $J_X^{\varphi}:=D\varphi \circ J_X D\varphi^{-1}$ for any $\varphi\in\map(X,X)$ (and similarly for $J_Y^{\psi}$.)  Notice that 
\[D\varphi^{-1}J_X^{\varphi}=J_X D\varphi^{-1}.\] Then we define the following:
\begin{defi}
	Let $(X,\omega_X,J_X)$, $(Y,\omega_Y,J_Y)$ be compact K{\"a}hler manifolds. Then we define 
	\[\Kmap_{\omega_Y}(X,Y;J_X):=\{f\in \diff(X,Y)| J_X^{f}=Df J_X Df^{-1}\in\cJ(Y,\omega_Y)\}.\]
	We denote $\Kmap(X,Y)=\Kmap_{\omega_Y}(X,Y;J_X)$ if there is no confusion on the K{\"a}hler form.  
\end{defi}
As a remark, we can also define $\Kmap(X,Y)$ by fixing $J_Y$ and moving $J_X$.
\begin{lemm}\label{lem: complex structure}
	The manifold $(\map(X,Y), J)$, where $J\delta f:=Df J_X Df^{-1}\delta f=J_X^f $, is a complex manifold.
\end{lemm}
\begin{proof}
	It is obvious that $(\map(Y,X),J_X)$ is a complex manifold as $J_X$ is integrable. Notice that the map $inv:\map(X,Y)\rightarrow \map(Y,X)$ defined by 
	\[inv(f)=f^{-1}\] is a diffeomorphism, hence we can consider $J$ on $\Kmap(X,Y)$ by \[inv_*J_X=D\, inv^{-1} J_X D\, inv.\] For any $v\in T_f\Kmap(X,Y)$, $D(inv(v))=Df^{-1} v \circ f^{-1} \in T_{f^{-1}}\Kmap(Y,X)$. Therefore, for any $w\in T_f^{-1}\Kmap(Y,X)$,
	\[inv_*J(v) = Df(J_XDf^{-1} v \circ f^{-1})\circ f=Df J_X Df^{-1}v. \] Therfore, $(\Kmap(X,Y),J)$ is a complex manifold with integrable almost complex structure $J$. 
\end{proof}
\begin{lemm}\label{lem: complex structure of Kmap}
	$\Kmap(X,Y)$ and $\Kmap(X,Y;p)^+$ is a complex submanifold. 
\end{lemm}
\begin{proof}
	Consider $\map(X,Y)\times \cJ(Y,\omega_Y)$ with the product complex structure $J(\sigma, A)=(J_X^f\sigma, J_YA)$ for all $(\sigma,A)\in T_{f,J_Y}\map(X,Y)\times \cJ(Y,\omega_Y)$. Then we have a subvariety \[\cW:=\{(f, J_Y)| Df(J_X^0)Df^{-1}=J_Y\},\] here $J_X^0$ is fixed. We can rewrite the relation as $f_*J_Y-J_X=0$ as an endmorphism. By this, we can consider the map $F:\map(X,Y)\times \cJ(Y,\omega_Y)\rightarrow End(TX)$ by 
	\[F(f, J_Y)=f_*J_Y-J_X.\] Then $\cW=\{F(f, J_Y)=0\}.$ When we differenate with respect to $J_Y$ direction, say $A$, then 
	\[\delta_A F(f, J_Y)=\left.\dfrac{d}{dt}\right|_{t=0}(Df^{-1}J_Y^tDf)=Df^{-1}ADf\] which is bounded and  indeed $c|A|<|\delta_A F(f, J_Y)|<C|A| $ for some $c,C$, and for any norm. So $\cW$ is locally a graph, which gives the smooth structure of $\cW$.
	
	We now show $\cW$ is a complex subvariety. Let $\dfrac{d}{dt}f_t=\sigma$ and $\dfrac{d}{dt}J_Y^t=A$ at $t=0$. Then the condition on tangent space is given by
	\[-Df^{-1} D\sigma Df^{-1} J_YDf+Df^{-1} A Df+Df^{-1}A D\sigma=0.\] We now see if the vector $(J_X^f\sigma, J_YA)$ with $J_X^f=J_Y$ satisfies this relation. Notice that $J_X^f=J_Y$ implies 
	\[J_X Df^{-1}=Df^{-1}J_Y.\] So
	\begin{align*}
	&-Df^{-1} DJ_X^f\sigma Df^{-1} J_YDf+Df^{-1} J_Y ADf+Df^{-1}J_Y DJ_X^f\sigma\\
	&=-J_XDf^{-1} D\sigma Df^{-1} J_YDf+J_XDf^{-1}  ADf+Df^{-1}J_X^f J_YD\sigma\\
	&=-J_XDf^{-1} D\sigma Df^{-1} J_YDf+J_XDf^{-1}  ADf+J_XDf^{-1} J_YD\sigma\\
	=&J_X (-Df^{-1} D\sigma Df^{-1} J_YDf+Df^{-1} A Df+Df^{-1}A D\sigma)\\
	=&0.
	\end{align*}
	Also, we need to show that the map $\pi: \cW\rightarrow \map(X,Y)$ is injective, holomorphic and the image is $\Kmap(X,Y)$. The  injectivity is obvious as if $\pi(f, J_Y)=\pi(f',J_Y')$, then $f=f'$. When $f=f'$, $J_X^f=J_X^{f'}$. By the definition, $J_Y\in \cJ(X,\omega_X)$, hence it is $\Kmap_{\omega_X}(X,Y)$. Finally it is holomorphic as this is the restriction of the projection map $\pi:\map(X,Y)\times \cJ(Y,\omega_Y)\rightarrow \map(X,Y)$ which is holomorphic. \\
	
	Notice that $\Kmap(X,Y;p)^+$ is an open subset of $\Kmap$. As $\omega_X^{n-p}\wedge\omega_Y^p$ is $J$ invariant for $(J_X,J_Y)\in \cJ(X,\omega_X)\times\cJ(Y,\omega_Y)$, so this is a complex submanifold.
\end{proof}
\begin{rema}
Using the same argument, we can prove that $\Kmap(X,Y)$ is a complex submanifold of $(\map(X,Y), J_Y)$ as well.  
\end{rema}

\begin{defi}
	Let $(X,\omega_X,J_X)$, $(Y,\omega_Y,J_Y)$ be compact K{\"a}hler manifolds. Then we define 
	\[\Kmap_{\omega_Y}(X,Y;J_X):=\{f\in \diff(X,Y)| J_X^{f}=Df J_X Df^{-1}\in\cJ(Y,\omega_Y)\}.\]
	We denote $\Kmap(X,X)=\Kmap_{\omega_X}(X,X,J_X)$ if there is no confusion on the choice of the K{\"a}hler form.  Then we define 
	\[\cX_{f,p}^+:=(\Kmap(X,X)\times \Kmap(Y,Y))\cap F_f^{-1}(\map(X,Y;p)^+).\] 
	Moreover, let $(v,w)\in T_{(\varphi,\psi)}(\map(X,X)\times\map(Y,Y))$ we also define an almost complex structure on $\map(X,X)\times\map(Y,Y)$ by
	\[J_{\map}(v,w):=(J_{X}^{\varphi}v,J_Y^{\psi}w).\]
\end{defi}

We now prove the main proposition in this section:
\begin{prop}\label{prop:  pseudo moment map}
	Let $f:X\rightarrow Y$ is a biholomorphism,	and define an action  $\ham(X,\omega_X)\times \ham(Y,\omega_Y)$ on $\map(X,X)\times \map(Y,Y)$ which is given by 
	\[(\sigma,\eta)\cdot (\varphi,\psi):=(\sigma\circ \varphi,\eta\circ\psi). \] Then
	\begin{enumerate}
		\item $F_f$ commutes with the group action.
		\item $\cX_{f,p}^+$ is a complex submanifold, 
		\item The action $\ham(X,\omega_X)\times \ham(Y,\omega_Y)$ is closed in $\cX_{f,p}^+$.
		\item  $F_f^*\Omega_p$ is  $J_{\map}$ invariant.
	\end{enumerate}
\end{prop}
\begin{proof}
	\begin{enumerate}
		\item $F_f(\sigma\circ \varphi,\eta\circ\psi)=\eta\circ\psi \circ f\circ\varphi^{-1}\sigma^{-1}=(\sigma,\eta)\cdot F_f(\varphi,\psi)$.
		\item To show $\cX_{f,p}^+$ is a smooth manifold, we only need to show $\Kmap(X,X)$ is smooth. As $F_f$ is continuous, so it implies $\cX_{f,p}^+$ is an open subset, hence it is smooth. \\
		
		Define  $G:\cJ(X,\omega_X)\times \Kmap(X,X)\rightarrow \End(\Gamma(TX))$ by 
		\[G(J,\varphi)=J_X^{\varphi}-J.\] Then $\Kmap(X,X)\cong\{(J,\varphi)|G(J,\varphi)=0.\}$ Also, let $A\in T_{J,\varphi}\cJ(X,\omega_X)$,  then
		\[DG_{J,\varphi}(A,0)=-Id,\] hence the implicit function theorem implies that there exists $H:U \subset \Kmap(X,X)\rightarrow V\subset \cJ(X,\omega_X)$ which for $G:U\times V\rightarrow \End(\Gamma(TX))$, we have 
		\[G(J,\varphi)=G(H(\varphi),\varphi)=0.\] Therefore, $(U,H)$ gives a local coordinate, which implies $\Kmap(X,X)$ is smooth. Hence $\Kmap(Y,Y)$ is also smooth, and thus $\cX_{f,p}^+$ is smooth.\\
		
		We now show $\cX_{f,p}^+$ is $J$ invariant. Again, as $F_f^{-1}(\map(X,Y;p))$ is open, and $\Kmap(X,X)$ and $\Kmap(Y,Y)$ have the same defining function, we only need to show $\Kmap(X,X)$ is $J_X^{\varphi}$ invariant. Then the argument can be used as showing $\Kmap(Y,Y)$ is also $J_Y^{\psi}$ invariant.  Suppose $\sigma\in T_{\varphi}\Kmap(X,X)$.  The equation we have is the following: for all $(v,w)\in T_xX$ 
		\[\varphi^*\omega_X(J_X v, J_Xw )=\varphi^*\omega_X(v, w).\]
		Differentiating it  along $v$, we get
		\[\omega_X(D\sigma J_X v,  D\varphi J_Xw)+\omega_X(D\varphi J_Xv, D\sigma J_X w)=\omega_X(D\sigma v, D\varphi w)+\omega_X(D\varphi v, D\sigma w).\]
		Now 
		\begin{align*}
		&\omega_X( J_X^{\varphi}D\sigma J_Xv, D\varphi J_Xw)+\omega_X(D\varphi J_Xv, J_X^{\varphi}D\sigma J_Xw)-\omega_X(J_X^{\varphi}D\sigma v,D\varphi w)-\omega_X(D\varphi v, J_X^{\varphi}D\sigma w)\\
		=&\omega_X( D\sigma J_Xv, J_X^{\varphi}D\varphi J_Xw)+\omega_X(J_X^{\varphi}D\varphi J_Xv, D\sigma J_Xw)-\omega_X(D\sigma v,J_X^{\varphi}D\varphi w)-\omega_X(J_X^{\varphi}D\varphi v,D\sigma w)\\
		=&\omega_X( D\sigma J_Xv, D\varphi J_X  J_Xw)+\omega_X( D\varphi J_X  J_Xv, D\sigma J_Xw)-\omega_X(D\sigma v, D\varphi J_X  w)-\omega_X( D\varphi J_X  v,D\sigma w)
		\end{align*}
		We let $u=J_Xw$, then $w=-J_Xu$, hence the expression becomes 
		\begin{align*}
		&\omega_X( D\sigma J_Xv, D\varphi J_X  u)-\omega_X( D\varphi v, D\sigma u)-\omega_X(D\sigma v, D\varphi u)+\omega_X( D\varphi J_X  v,D\sigma J_Xu)\\
		=&\omega_X( D\sigma J_Xv, D\varphi J_X  u)+\omega_X( D\varphi J_X  v,D\sigma J_Xu)-\omega_X( D\varphi v, D\sigma u)-\omega_X(D\sigma v, D\varphi u)=0.
		\end{align*}
		\item As $F_f$ preserves the group action,  we only need to show $\Kmap(X,X)\times \Kmap(Y,Y)$ is closed under the action. Let $(\sigma,\eta)\in \ham(X,\omega_X)\times \ham(Y,\omega_Y)$. Then 
		\[(\sigma\circ \varphi)^*\omega_X=\varphi^*\sigma^*\omega_X=\varphi^*\omega_X,\] hence $(\sigma\circ \varphi)^*\omega_X$ is $J_X$-invariant. Similarly, $(\eta \circ\psi)^*\omega_Y$ is $J_Y$-invariant.
		\item  For $(v,w), (v',w')\in T_{\varphi,\psi}\Kmap(X,X)\times \Kmap(Y,Y)$, 
		\[D{F_p}_{\varphi,\psi}(v,w)|_x=w|_{f\circ \varphi^{-1}(x)}-D\psi Df D\varphi^{-1} v|_{\varphi^{-1}(x)},\] so
		\begin{align*}
		& F_p^*\Omega_p(J_{\map}(v,w),J_{\map}(v',w')) &\\
		=& F_p^*\Omega_p(J_X^{\varphi}v, J_Y^{\psi}w),(J_X^{\varphi}v', J_Y^{\psi}w') &\\
		=&
		\Omega_p(J_Y^{\psi}w-D\psi Df D\varphi^{-1} J_X^{\varphi}v, J_Y^{\psi}w'-D\psi Df D\varphi^{-1} J_X^{\varphi}v') &\\
		=& \Omega_p(J_Y^{\psi}w-D\psi Df J_XD\varphi^{-1} v, J_Y^{\psi}w'-D\psi Df J_XD\varphi^{-1} v') & (\because D\varphi^{-1} J_X^{\varphi}=J_X D\varphi)\\
		=& \Omega_p(J_Y^{\psi}w-D\psi J_YDf D\varphi^{-1} v, J_Y^{\psi}w'-D\psi J_YDf D\varphi^{-1} v') & (\because J_Y Df=Df J_X)\\
		=& \Omega_p(J_Y^{\psi}(w-D\psi Df D\varphi^{-1} v), J_Y^{\psi}(w'-D\psi Df D\varphi^{-1} v')) & (\because D\psi J_Y=J_Y^{\psi} D\psi)\\
		=& \int_X \omega_Y(J_Y^{\psi}(w-D\psi Df D\varphi^{-1} v), J_Y^{\psi}(w'-D\psi Df D\varphi^{-1} v'))\omega_X^{n-p}\wedge\omega_Y^p & \\
		=& \int_X \omega_Y((w-D\psi Df D\varphi^{-1} v), (w'-D\psi Df D\varphi^{-1} v'))\omega_X^{n-p}\wedge\omega_Y^p & ( \because J_Y^{\psi}\in \cJ(Y,\omega_Y))\\
		=&  F_p^*\Omega_p((v,w),(v',w')).
		\end{align*} 
		Hence it is $J_{\map}$-invariant. 
	\end{enumerate}
\end{proof}
As $\cX_{f,p}^+$ is a complex manifold, we observe that if $(v,w)\in T_{(\varphi,\psi)}(\cX_{f,p}^+)$, then  $J_{\map}(v,w)=(J_X^{\varphi}v,J_Y^{\psi}w)\in T_{(\varphi,\psi)}(\cX_{f,p}^+)$. As $\omega_Y(J_Y^{\psi}u,u)>0$ if $u\neq 0$, so for any $(v,w)$, if 
\[w-D\psi Df D\varphi^{-1}v\neq 0,\] then we can choose 
\[(v',w')=-J_{\map}(v,w).\] However, if $w=D\psi Df D\varphi^{-1}v$, then it is degenerate. Indeed, the problem is $F_f$ may not be injective. Indeed, if we consider $X=Y$, $\omega_X=\omega_Y$, then $f=id$ solve the problems, but for any $\sigma\in \ham^{\CC}(X)$, $(\sigma, \sigma)$ will solve the equation as well. Therefore, we cannot apply the theory directly.

As  $\cX_{f,p}^+$ is closed under the action of  $\ham(X,\omega_X)\times\ham(Y,\omega_Y)$ and $F_f$ preserves the action, we can still consider the orbit space \[\cO_f:=\ham^{\CC}(X,\omega_X)\times\ham^{\CC}(Y,\omega_Y)\cdot (id,id)\subset (\cX_{f,p}^+)\] (as we mentioned before, $\ham^{\CC}(X,\omega_X)\times\ham^{\CC}(Y,\omega_Y)$ is not a group). Notice that it may not be a manifold, but only a complex variety.

 Indeed, suppose there exists $\sigma\in \ham^{\CC}(X,\omega_X)\cap\ham^{\CC}(Y,\omega_Y)$, and $\omega_X$, $\omega_Y$ solved coupled equation $p$, then $\sigma^*\omega_X$, $\sigma^*f^*\omega_Y$ also solved equation $p$. This example exists, say,
\begin{exam}
Consider $(X,\omega_0,\omega_1)$ with 
\[[\omega_0]=[\omega_1].\] Then by definition, there exists $\sigma\in \ham^{\CC}(X,\omega_0)$ such that 
\[\sigma^*\omega_1=\omega_0.\] Then 
\[\omega_0^{n-p-1}\wedge\sigma^*\omega_1^{p+1}=\omega_0^n; {\ } \sigma_*\omega_0^{n-p}\wedge\omega_1^p=\omega_1^n,\] that is $(id,\sigma)\in \cO_{id}$ solves the equation. Moreover, for any $\eta\in \ham^{\CC}(X,\omega_0)=\sigma^*\ham^{\CC}(X,\omega_1)$, $(\eta,\sigma\eta)\in\cO_{id}$ and
\[\sigma\circ\eta\circ\eta^{-1}=\sigma \] implies it also solves the same moment map equation.  
\end{exam} 
Notice that we can consider the equivalent class, namely, 
\[(\sigma,\eta)\sim (\sigma',\eta')\text{ if }F_f((\sigma,\eta))=F_f((\sigma',\eta')),\] that is,
\[\eta'\circ f\circ \sigma'^{-1}=\eta\circ f\circ \sigma^{-1}.\] 
Notice that we may simply consider $[\cO_{f}]\subset \map(X,Y;p)^+$. Hence we can restrict the moment map into  $[\cO_{f}]$ if it is a manifold.

\begin{coro}
 Let $(X,\omega_X)$	and $(Y,\omega_Y)$ be two K{\"a}hler manifolds with two K{\"a}hler forms, and $f$ is a biholomorphism. Suppose $[\cO_f]$ is a manifold, then
$\displaystyle \mu_{p}:[\cO_f]\rightarrow Lie\left(\prod_{i=0}^k\ham(X_i,\omega_i)\right)^*$ is a moment map. In particular, if $\ham(X,\omega_X)\cap f^*\ham(Y,\omega_Y)={id}$ , then $\mu_p$ is well defined. 
\end{coro} 
\begin{proof}
Notice that $[\cO_f]\subset \map(X,Y;p)^+$, and it is closed under the action. Hence $\mu_p|_{[\cO_f]}$ is well defined.
\end{proof} 
\section{moment map for embedding}
In the previous theory, we always assume $X=Y$ as a same K\"{a}hler manifold, and $f_0=id$. We now provide a case that  $X$ and $Y$ are  not diffeomorphic.  

Let $(X,\omega_X),(Y,\omega_Y)$ be two symplectic manifolds with dimensions $n,m$, where $n\leq m$.  Define $\Emap(X,Y)$ to be the space of embedding maps and 
\[\Emap(X,Y;p)^+:=\{f\in\Emap(X,Y)| \omega_X^{n-p}\wedge f^*\omega_Y^p>0\}.\]
Notice that $f^{-1}$ is well defined on $f(X)$ and for this case, $f_*=f^{-1}$ on $f(X)$. 
Let $Z\subset Y$ be a $k$ dimensional submanifold. Then we denote $\delta_{Z}$ be the $m-k$ current on $Y$, which 
\[\int_Y\delta_Z\wedge\alpha:=\int_Z\alpha\] for all $k$ forms $\alpha$ on $Y$.   
\begin{lemm}	
	Let $(X,\omega_X), (Y,\omega_Y)$ are two symplectic manifolds with finite volume with respect to $\omega_X, \omega_Y$, and let $0\leq p\leq n-1$. Then the moment map 
	$\mu_p:\Emap(X,Y;p)^+\rightarrow Lie(\ham(X,\omega_X)\times\ham(Y,\omega_Y))^*$ is given by \[\mu_p(f)=\left(\dfrac{n}{n-p}\left(c_1\dfrac{\omega_X^n}{n!}-\dfrac{\omega_X^{n-p-1}\wedge f^*\omega_Y^{p+1}}{(n-p-1)!(p+1)!}\right), \dfrac{m}{m-p}\left( c_2\dfrac{\omega_Y^m}{m!}-\delta_{f(X)}\wedge\left(f_*(\omega_X^{n-p}\wedge f^*\omega_Y^p)\right)\right)\right).\]
\end{lemm}
\begin{proof}
	The proof is basically the same as the proof of theorem \ref{thm: main theorem of moment map}. The main difference is that $\Emap(X,Y)$ and $\map(Y,Y)$ is not a one-one correspondence. However, given $v'\in T_{f}\Emap(X,Y)$, $v'|_x\in T_{f(x)}Y$. Hence, we can still identify it as $v\in T_{g}\map(f(X),Y)$, where $g_t(y):= f_t\circ f^{-1}(y)$. After that, we extend this $g_t$ to $\hat{g_t}:Y\rightarrow Y$. Then the same proof can be applied. 
	\begin{align*}
	 &\left.\dfrac{d}{dt}\right|_{t=0}	 \int_{f_t(X)}\psi(y) \dfrac{{f_{t}}_*(\omega_X^{n-p}\wedge f_t^*\omega_Y^p)}{(n-p)!p!}\\
	 =&  \left.\dfrac{d}{dt}\right|_{t=0} \int_X \psi(f_t(x))\dfrac{\omega_X^{n-p}\wedge f_t^*\iota_v\omega_Y^p}{(n-p)!p!}\\
	 =& \int_X d\psi (v\circ f)  \dfrac{\omega_X^{n-p}\wedge f_t^*\iota_v\omega_Y^p}{(n-p)!p!}+p\int_X \psi(f(x))\dfrac{df^*\iota_v\omega_Y\wedge\omega_X^{n-p}\wedge f_t^*\iota_v\omega_Y^{p-1}}{(n-p)!(p)!}\\
	 =&\int_X \omega_Y(\xi_{\psi}, (v\circ f))  \dfrac{\omega_X^{n-p}\wedge f_t^*\iota_v\omega_Y^p}{(n-p)!p!}-\dfrac{p}{m}\int_{f(X)} \omega_Y(\xi_{\psi}, (v))  \dfrac{\omega_X^{n-p}\wedge f_t^*\iota_v\omega_Y^p}{(n-p)!p!} \\
	 =&\dfrac{m-p}{m}\int_X \omega_Y(\xi_{\psi}, (v\circ f))  \dfrac{\omega_X^{n-p}\wedge f_t^*\iota_v\omega_Y^p}{(n-p)!p!}.
	\end{align*}
	 Notice that this is independent of the choice of extension of $g$ as the term $f_t^*\iota_v\omega_Y$ only depends on $v$ and $f_t$, but not the extension $\hat{v}:=\dot{\hat{g_t}}|_{t=0}$.
\end{proof}
\begin{rema}\label{rem: notation of moment map p}
Notice that as $p$ is fixed, we can take $\Omega_X':=\dfrac{n}{(n-p)}\Omega_X$, $\Omega_Y'=\dfrac{m}{m-p}\Omega_Y$ to remove the leading coefficient. We will denote this moment map as $\mu_p$ from now on.
\end{rema}
\begin{rema}
When $Y$ is compact, given any function $\psi$, 
\[\psi-\dfrac{1}{Vol(Y)}\int_Y \psi \dfrac{\omega_Y^m}{m!}\in Lie(\ham(Y,\omega_Y)).\] However, for the case where $Y$ is non compact, and $\displaystyle\int_Y\dfrac{\omega_Y^m}{m!}=\infty$, we cannot normalized $\psi$. So we need to assume $\displaystyle\int_Y\psi\omega_Y^m=0$.  
\end{rema} 
\begin{rema}
For $p=n$,  the map $\mu_n(f): \Emap(X,Y;p)^+\rightarrow Lie(\ham(Y,\omega_Y);n)^*$ is a moment map. Therefore, we can still get a non-trivial moment map for $p=n$ if this is an embedding. 
\end{rema}

\appendix
\section{Analytic compuation of convexity of $\cM_{\cJ,p}$}
In this section, we will show that the Mabuchi functional $\cM_{\cJ,p}$ is stictly convex along the smooth geodesic $(h_{0,t},\cdots, h_{k,t})$, where the geodesic equation is given by 
\[\ddot{h}_{i,t}-|\nabla \dot{h}_{i,t}|_{\omega_{i,h_i}}^2=0\] 
 for all $0\leq i\leq k$.
As the standard Mabuchi functional $\cM_{\cJ_p}$ is well known to be convex, it suffices to consider $\cM_{p}=\cM_{\cJ,p}-\cM_{\cJ}$. For simplicity, we will only consider the case $k=1$. We also denote $\varphi_{i,t}=\dot{h_{i,t}}$, and   
$\omega^{[k]}:= \dfrac{\omega^k}{k!}$. As 
\[|\nabla \varphi_i|_{\omega_i}^2=\dfrac{\omega_i(X_{\varphi_i}, JX_{\varphi_i})}{\omega_i^{[n]}},\]
The geodesic equation with Lemma \ref{lem: weak interior product commute} implies that 
\[|\nabla \varphi_i|_{\omega_i}^2\omega_i^{n-p}\wedge \alpha^{p}=\sqrt{-1}n\dd \varphi\wedge \dbar \varphi_i\wedge\omega_i^{n-p-1}\wedge \alpha^{p}. \] Hence 
\[\dfrac{n-p}{n}\int_X|\nabla \varphi_i|_{\omega_i}^2\omega_i^{[n-p]}\wedge \alpha^{[p]}=-\int_X \varphi\wedge \ddbar \varphi_i\wedge\omega_i^{[n-p-1]}\wedge \alpha^{[p]}. \] 
As 
\[d\cM_{p}(\varphi_{0,t},\varphi_{1,t}):=\int_X \varphi_{0,t}(\omega_{0,t}^{[n-p-1]}\wedge \omega_{1,t}^{[p+1]}-c_1\omega_{0,t}^{[n]})+\int_X \varphi_{1,t}(\omega_{0,t}^{[n-p]}\wedge \omega_{1,t}^{[p]}-c_2 \omega_{1,t}^{[n]}),\]
\begin{align*}
&\dfrac{d^2}{dt^2}\cM_{p}(h_{0,t}, h_{1,t})\\
=& \int_X \dot{\varphi}_{0,t}(\omega_{0,t}^{[n-p-1]}\wedge \omega_{1,t}^{[p+1]}-c_1\omega_{0,t}^{[n]})+\int_X \dot{\varphi}_{1,t}(\omega_{0,t}^{[n-p]}\wedge \omega_{1,t}^{[p]}-c_2 \omega_{1,t}^{[n]})\\
&+\int_X \varphi_{0,t}(\ddbar \varphi_{0,t}\wedge\omega_{0,t}^{[n-p-2]}\wedge \omega_{1,t}^{[p+1]}-c_1\ddbar \varphi_{0,t}\wedge\omega_{0,t}^{[n-1]})\\
&+\int_X \varphi_{1,t}(\ddbar \varphi_{0,t}\wedge \omega_{0,t}^{[n-p-1]}\wedge \omega_{1,t}^{[p]})+\int_X \varphi_{0,t}(\ddbar \varphi_{1,t}\wedge\omega_{0,t}^{[n-p-1]}\wedge \omega_{1,t}^{[p]})\\
&+\int_X \varphi_{1,t}(\ddbar \varphi_{1,t}\wedge \omega_{0,t}^{[n-p]}\wedge \omega_{1,t}^{[p-1]}-c_2\ddbar \varphi_{1,t}\wedge \omega_{1,t}^{[n-1]})\\
=& \int_X \dot{\varphi}_{0,t}(\omega_{0,t}^{[n-p-1]}\wedge \omega_{1,t}^{[p+1]}-c_1\omega_{0,t}^{[n]})+\int_X \dot{\varphi}_{1,t}(\omega_{0,t}^{[n-p]}\wedge \omega_{1,t}^{[p]}-c_2 \omega_{1,t}^{[n]})\\
&-\int_X |\nabla \varphi_{0,t}|_{\omega_{0,t}}^2(c_1\omega_{0,t}^{[n]}-\dfrac{n-p-1}{n}\omega_{0,t}^{[n-p-1]}\wedge \omega_{1,t}^{[p+1]})\\
&-\int_X |\nabla\varphi_{1,t}|_{\omega_{1,t}}^2( \dfrac{p}{n}\omega_{0,t}^{[n-p]}\wedge \omega_{1,t}^{[p]}-c_2 \omega_{1,t}^{[n]})\\
&+\int_X \varphi_{1,t}(\ddbar \varphi_{0,t}\wedge \omega_{0,t}^{[n-p-1]}\wedge \omega_{1,t}^{[p]})+\int_X \varphi_{0,t}(\ddbar \varphi_{1,t}\wedge\omega_{0,t}^{[n-p-1]}\wedge \omega_{1,t}^{[p]})\\
=&\dfrac{p+1}{n}\int_X |\nabla \varphi_{0,t}|_{\omega_{0,t}}^2\omega_{0,t}^{[n-p-1]}\wedge \omega_{1,t}^{[p+1]}+\dfrac{n-p}{n}\int_X |\nabla \varphi_{1,t}|_{\omega_{1,t}}^2 \omega_{0,t}^{[n-p]}\wedge \omega_{1,t}^{[p]}\\
&+\int_X \varphi_{1,t}(\ddbar \varphi_{0,t}\wedge \omega_{0,t}^{[n-p-1]}\wedge \omega_{1,t}^{[p]})+\int_X \varphi_{0,t}(\ddbar \varphi_{1,t}\wedge\omega_{0,t}^{[n-p-1]}\wedge \omega_{1,t}^{[p]})\\
=&\dfrac{(p+1)\sqrt{-1}}{n-p-1}\int_X\dd \varphi_{0,t}\wedge \dbar  \varphi_{0,t} \wedge \omega_{0,t}^{[n-p-2]}\wedge \omega_{1,t}^{[p+1]}+\dfrac{\sqrt{-1}(n-p)}{p}\int_X \dd \varphi_{1,t} \wedge \dbar \varphi_{1,t} \wedge \omega_{0,t}^{[n-p]}\wedge \omega_{1,t}^{[p-1]}\\
&-\sqrt{-1}(\int_X \dd\varphi_{1,t}\wedge \dbar \varphi_{0,t}\wedge \omega_{0,t}^{[n-p-1]}\wedge \omega_{1,t}^{[p]}+\int_X \dd\varphi_{0,t}\wedge \dbar \varphi_{1,t}\wedge\omega_{0,t}^{[n-p-1]}\wedge \omega_{1,t}^{[p]})
\end{align*}
Using the same proof as in Lemma \ref{lem: weak interior product commute}, and 
\[\sqrt{-1} \dd \varphi_{i,t}\wedge \dbar \varphi_{j,t}=d \varphi_{i,t}\wedge d^c\varphi_{j,t}=\omega_{i,t}(X_{\varphi_{i,t}}, \bullet)\wedge \omega_{j,t}(-JX_{\varphi_{j,t}},\bullet),\] the expression becomes
\begin{align*}
&\dfrac{1}{(n-p-1)!p!}\int_X \omega_{0,t}(X_{\varphi_{0,t}}, \bullet)\wedge \omega_{0,t}(-JX_{\varphi_{0,t}}, \bullet)\wedge \omega_{0,t}^{n-p-2}\wedge \omega_{1,t}^{p+1}\\
&+\dfrac{1}{(n-p-1)!p!}\int_X \omega_{1,t}(X_{\varphi_{1,t}}, \bullet)\wedge \omega_{1,t}(-JX_{\varphi_{1,t}}, \bullet)\wedge \omega_{0,t}^{n-p}\wedge \omega_{1,t}^{p-1}\\
&-\dfrac{2}{(n-p-1)!p!}\int_X \omega_{1,t}(X_{\varphi_{1,t}}, \bullet)\wedge \omega_{0,t}(-JX_{\varphi_{0,t}}, \bullet)\wedge \omega_{0,t}^{n-p-1}\wedge \omega_{1,t}^p\\
=&\dfrac{1}{(n-p-1)!p!}\int_X \omega_{0,t}(X_{\varphi_{0,t}}, JX_{\varphi_{0,t}})\wedge \omega_{1,t}\wedge \wedge \omega_{0,t}^{n-p-1}\wedge \omega_{1,t}^p\\
&+\dfrac{1}{(n-p-1)!p!}\int_X \omega_{1,t}(X_{\varphi_{1,t}}, JX_{\varphi_{1,t}})\wedge  \omega_{0,t}^{n-p-1}\wedge \omega_{1,t}^p\\
&-\dfrac{2}{(n-p-1)!p!}\int_X \omega_{1,t}(X_{\varphi_{1,t}}, JX_{\varphi_{0,t}})\wedge \omega_{0,t}^{n-p}\wedge \omega_{1,t}^p.
\end{align*}
We claim that if $\alpha, \beta$ are two forms, then 
\[\iota_{v}\iota_{w}\alpha\wedge\beta= \alpha\wedge \iota_{v}\iota_{w}\beta=\iota_{v}\iota_{w}\beta\wedge \alpha.\] With this claim, and 
\[\omega_{1,t}(X_{\varphi_{1,t}}, JX_{\varphi_{0,t}})=\omega_{1,t}(X_{\varphi_{0,t}}, JX_{\varphi_{1,t}}),\] the expression becomes
\begin{align*}
&\dfrac{1}{(n-p-1)!p!}\int_X \omega_{1,t}(X_{\varphi_{0,t}}, JX_{\varphi_{0,t}})\wedge \omega_{0,t}\wedge \wedge \omega_{0,t}^{n-p-1}\wedge \omega_{1,t}^p\\
&+\dfrac{1}{(n-p-1)!p!}\int_X \omega_{1,t}(X_{\varphi_{1,t}}, JX_{\varphi_{1,t}})\wedge  \omega_{0,t}^{n-p-1}\wedge \omega_{1,t}^p\\
&-\dfrac{2}{(n-p-1)!p!}\int_X \omega_{1,t}(X_{\varphi_{1,t}}, \bullet)\wedge \omega_{0,t}(-JX_{\varphi_{0,t}}, \bullet)\wedge \omega_{0,t}^{n-p-1}\wedge \omega_{1,t}^p\\
=&\dfrac{1}{(n-p-1)!p!}\int_X \omega_{1,t}(X_{\varphi_{0,t}}-X_{\varphi_{1,t}}, J(X_{\varphi_{0,t}}-X_{\varphi_{1,t}}))  \omega_{0,t}^{n-p}\wedge \omega_{1,t}^p
\end{align*}

We finally show the claim. 
To show that, observe that 
\[(\iota_v\iota_w \alpha) \wedge \beta=\iota_v (\iota_w \alpha\wedge \beta)+\iota_w\alpha \wedge \iota_v\beta=\iota_v ((\iota_w \alpha)\wedge \beta)+\iota_w(\alpha \wedge (\iota_v\beta))- \alpha\wedge \iota_w\iota_v\beta.\] Therefore, 
\[(\iota_v\iota_w \alpha) \wedge \beta-\alpha\wedge \iota_v\iota_{w}\beta=\iota_v ((\iota_w \alpha)\wedge \beta)+\iota_w(\alpha \wedge (\iota_v\beta)).\] Also, 
\[(\iota_v\iota_w \alpha) \wedge \beta-\alpha\wedge \iota_v\iota_{w}\beta=(\iota_w\iota_v \beta) \wedge \alpha-\beta\wedge \iota_w\iota_{v}\alpha=\iota_w ((\iota_v \beta)\wedge \alpha)+\iota_v(\beta \wedge (\iota_w\alpha)). \] as 
\[\iota_w ((\iota_v \beta)\wedge \alpha)+\iota_v(\beta \wedge (\iota_w\alpha))=-\iota_w (\alpha\wedge (\iota_v\beta) )-\iota_v( (\iota_w\alpha)\wedge \beta ),\] 
\[2((\iota_v\iota_w \alpha) \wedge \beta-\alpha\wedge \iota_v\iota_{w}\beta)=((\iota_v\iota_w \alpha) \wedge \beta-\alpha\wedge \iota_v\iota_{w}\beta)+(\iota_w\iota_v \beta) \wedge \alpha-\beta\wedge \iota_w\iota_{v}\alpha=0.\] 

\begin{rema}
We could point out that from this definition, we can see that $\cM_{p}$ is not strictly convex when $X_{\varphi}=X_{\phi}$. Hence we can see that $\mu_p$ is not a moment map unless we mod out this relation. 
\end{rema}

\end{document}